\documentclass{article}
\usepackage{amsthm}

\usepackage[english]{babel}
\usepackage[utf8x]{inputenc}
\usepackage[T1]{fontenc}
\usepackage{graphicx}

\usepackage[version=0.96]{pgf}
\usepackage{tikz}
\usetikzlibrary{arrows,shapes,decorations,automata,backgrounds,petri,cd}

\usepackage{algorithm2e}

\usepackage{amsmath,amssymb,amsfonts,bbm}

\usepackage{float}

\usepackage[justification=centering]{caption}

\usepackage{mathtools}
\usepackage{listings}

\newtheorem{thm}{Theorem}[section]
\newtheorem{prop}[thm]{Proposition}

\newtheorem{lemme}[thm]{Lemma}

\newtheorem{define}[thm]{Definition} 
\newtheorem{def-prop}[thm]{Definition-Proposition}

\newtheorem{ex}[thm]{Example}
\newtheorem{rem}[thm]{Remark}

\renewcommand\P{\mathcal P}

\newcommand\M{\mathcal{M}}
\newcommand\N{\mathbb N}

\newcommand\R{\mathbb R}
\newcommand\C{\mathbb C}
\newcommand\D{\mathcal D}

\newcommand\J{\mathcal{J}}

\newcommand\transp[1]{\, {}^t \! {#1}}

\newcommand\abs[1]{\left|#1\right|}
\newcommand\norm[1]{\left\|#1\right\|}

\newcommand\pref{\operatorname{Pref}}
\newcommand\pruneinf{\operatorname{prune}^\infty}

\newcommand\floor[1]{\left\lfloor #1 \right\rfloor}

\renewcommand\D{\mathcal D}
\newcommand\defi[1]{\textbf{#1}}

\newcommand\0{{\bf{0}}}
\newcommand\1{{\bf{1}}}
\renewcommand\P{\mathbb{P}}
\newcommand\Pc{\mathcal{P}}

\newcommand\Lrel{L^{\operatorname{rel}}}

\usepackage{hyperref}

\usepackage{pdfpages}

\tikzset{every state/.style={minimum size=0pt}}

\author{Paul MERCAT}
\title{Computation of invariant densities for continued fraction algorithms}

\begin{document}


\maketitle

\begin{abstract}
    We introduce the notion of matrices graph, defining continued fraction algorithms
    where the past and the future are almost independent. We provide an algorithm to convert more general
    algorithms into matrices graphs.
    We present an algorithm that computes exact invariant densities of certain continued fraction algorithms,
    including classical ones and some of their extensions.
    For finite extensions of the classical additive algorithm with two coordinates, we provide a more precise algorithm that
    decides whether the invariant density is composed of rational fractions and computes it.
    For any finite set of quadratic numbers, we construct a continued fraction algorithm
    whose invariant density are rational fractions containing the quadratic numbers.
\end{abstract}

\tableofcontents

\section{Introduction}

Usually, an additive continued fraction algorithm is a map from $\R_+^{d}$ to $\R_+^{d}$ which is linear by pieces,
and usually the pieces are polyhedra.
In this paper, we consider continued fraction algorithms given by matrices graph, which is more restrictive since it corresponds to algorithms where the past and the future are almost independent.
But we will see that classical algorithms can be converted to matrices graph (see Section~\ref{sec:general_to_matrices}),
and this approach offers several advantages. It is easier to describe the natural extension,
allows for more regular invariant densities, simplifies the conversion to a win-lose graph
for which we have an algorithmic criterion to check ergodicity (see~\cite{Fougeron}).
See Section~\ref{sec:exs} for some examples.

In~\cite{AL}, Arnoux-Labbe compute invariant densities by considering a natural extension,
constructed ad hoc for each example, following ideas from~\cite{AN}.
This article provides, for the first time, an algorithm that completely automates this technique.
It may fail in general, but it works for every classical continued fraction algorithm for which the invariant density is known (see Section~\ref{sec:exs}) and for some extensions of it.
For extensions of the classical additive algorithm with two coordinates, the algorithm is comprehensive:
it always provides a description of the invariant density.
It decides whether it is composed of rational fractions and computes it.
When they are not rational fractions, it provides them as a countable sum of rational fractions.
All algorithms described in this article are implemented for the Sage math computing system
and are freely available as a GitLab package (see Subsection~\ref{ss:sage} for more details).

In Section~\ref{sec:settings}, we define the notion of matrices graph and other concepts used in the paper.
Then, in Section~\ref{sec:algo:density}, we present the algorithm that computes the invariant density for a continued fraction algorithm associated with a given matrices graph.
In Section~\ref{sec:general_to_matrices}, we introduce an algorithm that convert more general continued fractions algorithms into matrices graphs.
In Section~\ref{sec:wl2}, we consider the particular case of win-lose graphs on two letters and provide an algorithm that decides whether the invariant density is a rational fraction.
We present many examples in Section~\ref{sec:exs}.
We conclude in Section~\ref{sec:quadratic} by providing an algorithm that takes a finite set of quadratic numbers as input
and returns a win-lose graph on two letters whose invariant density includes every quadratic number.

\section{Settings} \label{sec:settings}

\subsection{Matrices graphs and win-lose graphs} \label{ss_cfa}


A \defi{matrices graph} is a finite oriented graph,
labeled by matrices of size $d \times d$, such that from each vertex $s$,
if the outgoing edges are labeled by matrices $\M_s = \{m_1$, ..., $m_k\}$,
then we have $\R_+^d = \bigcup_{m \in \M_s} m \R_+^{d}$, and the union is quasi-disjoint.
In particular, such a graph defines a deterministic automaton if we define some initial state and final states.
We denote $i \xrightarrow{m} j$ an edge (also called transition) in the graph.

A matrices graph gives rise to a map $F : \R_+^{d} \times S \to \R_+^{d} \times S$, almost everywhere defined by
\[
    F(x,i) = (m^{-1}x, j) \text{ if } i \xrightarrow{m} j \text{ such that } x \in m\R_+^d,
\]
where $S$ is the set of vertices (also called states) of the graph.
We refer to this map as the \defi{continued fraction algorithm} associated with the matrices graph.
It is well-defined almost everywhere, and we can iterate it infinitely often on a set of full Lebesgue measure
(everywhere except a countable union of hyperplanes).
Note that the only information retained from the past is the element from the set $S$.

A \defi{win-lose graph} (also called \defi{simplicial system} in~\cite{Fougeron}) is another way to describe a continued fraction algorithm
that is more restrictive.
It is an oriented graph labeled by integers $\{0,...,d-1\}$ such that, from each vertex,
there exists at most one edge labeled with each integer.

The associated \defi{continued fraction algorithm} is the map $F : \R_+^{d} \times S \to \R_+^{d} \times S$ almost everywhere defined by
\[
    F(x, i) = (x', k) \text{ if } i \xrightarrow{j} k \text{ such that } \forall i \xrightarrow{l} k', x_j \leq x_{l}
\]
where $x'$ is the vector obtained from $x$
by subtracting $x_j$ from every $x_l$ such that $i \xrightarrow{l} k'$ with $l \neq j$ and $k' \in S$.
A win-lose graph can be seen as a matrices graph, where the matrix corresponding to an edge $i \xrightarrow{j} k$ is 
$I_d + \sum_{i \xrightarrow{l} k', l \neq j} E_{j,l}$, where $E_{j,l}$ is the matrix with a $1$ in coordinate $j,l$
and $0$ at every other coordinate.

The \defi{continued fraction expansion} (or simply \defi{expansion}) of a point $(x,i) \in \R_+^d \times S$,
is the sequence of matrices $(m_n)_{n \in \N}$ that appears when iterating the continued fraction algorithm from $(x,i)$.
We say that a continued fraction algorithm is \defi{convergent} if, for almost every $x \in \R_+^d$,
its expansion $(m_n)_{n \in \N}$ satisfies $\bigcap_{n \in \N} m_0 ... m_n \R_+^d = \R_+ x$.

\subsection{Invariant density}

In this subsection, we consider a matrices graph with vertices $S$.

We say that $\mu$ is an \defi{invariant measure} on $\R_+^{d} \times S$
if, for every measurable set $E$, we have $\mu(F^{-1}E) = \mu(E)$, and for every $\alpha \in \R_+$, $\mu(\alpha E) = \mu(E)$,
where $\alpha E = \{(\alpha x,i) \in \R_+^d \times S \mid (x,i) \in E \}$.

If an invariant measure is absolutely continuous with respect to the Lebesgue measure on $\R_+^{d} \times S$,
then its density function is called an \defi{invariant density}.

Let $f : \R_+^{d} \times S \to \R_+$ be an invariant density.
Then, for every $\alpha \in \R_+$, $x \in \R_+^d$, and $i \in S$, $f(\alpha x, i) = \frac{1}{\alpha^d} f(x, i)$.
We call \defi{invariant density at state  $i \in S$} the map $f_i : \R_+^{d} \to \R_+$ defined by $f_i(x) := f(x,i)$.

An invariant density satisfies, for every state $j \in S$ and almost every $x \in \R_+^{d}$, the relation
\[
    f_j(x) = \sum_{i \xrightarrow{m} j} \abs{\det(m)} f_i(m x).
\]

Usually densities are expressed as functions restricted to the standard simplex $\{x \in \R_+^{d} \mid (1,...,1)x = 1 \}$.
In that case, there is a Jacobian determinant of the map $x \mapsto mx/\norm{mx}_1$ that appears in such functional relations.
However, it is enough to take the restriction of functions $f_j$ to the standard simplex to get usual densities,
and it avoids making an arbitrary choice to parameterize the standard simplex.

We say that a continued fraction algorithm is \defi{ergodic} if there exists a unique ergodic
invariant measure absolutely continuous with respect to Lebesgue.

\subsection{Rational languages and limit sets}

A \defi{automaton} is a finite graph, labeled by letters in a finite alphabet,
with some state (or vertex) called \defi{initial state}, and a set of \defi{final states}.
In this article, we denote initial states by thick circles, and final states by double circles.
The \defi{language recognized by an automaton} is the set of finite words labeling a path
from the initial state to a final state.
We denote $s_0 \xrightarrow{u_1} ... \xrightarrow{u_n} s_n$ if there is a path in the automaton
labelled by a word $u_1...u_n$ from $s_0$ to $s_n$.

A \defi{rational language} (or \defi{regular language}) is a language recognized by an automaton.
Any rational language is recognized by a \defi{deterministic automaton}, which is an automaton such that
if $s \xrightarrow{i} s'$ and $s \xrightarrow{i} s''$, then $s' = s''$.

The set of rational languages is stable under many operations: union, intersection, complement, mirror,
image by a morphism, inverse image by a morphism, and Kleene star.
See for example~\cite{Carton} for more details.
The set of rational languages is also stable under prefixes: if $L$ is a rational language over an alphabet $A$, then
\[
    \pref(L) = \{ u \in A^* \mid \exists v \in A^*,\ uv \in L \}
\]
is rational. Indeed, if a pruned automaton recognizes $L$, then the same automaton where every state is final recognizes $\pref(L)$.

Let $L$ be a rational language over an alphabet $A \subset M_{d}(\R_+)$ of matrices.
We define the \defi{limit set} of $L$ by
\[
    \Lambda_L = \bigcap_{n \in \N} \bigcup_{k \geq n} \bigcup_{u_1...u_k \in L} u_1...u_k \R_+^{d}.
\]
In other words, the limit set is the set of vectors $x$ such that $x \in u_1..u_n \R_+^{d}$
for infinitely many words $u_1...u_n \in L$.

In the particular case where $L$ is stable under prefixes, we have
\[
    \Lambda_L = \bigcap_{n \in \N} \bigcup_{u_1...u_n \in L} u_1...u_n \R_+^{d},
\]
and $\Lambda_L$ is a closed set.

If $B$ and $C$ are two languages over alphabets included in $M_{d}(\R_+)$, then
\[
    \Lambda_{B \cup C} = \Lambda_B \cup \Lambda_C, \quad \text{ and } \quad
    \bigcup_{w \in B} w \Lambda_C \subseteq \Lambda_{BC} \subseteq \Lambda_B \cup \bigcup_{w \in B} w \Lambda_C.
\]

\subsection{Natural extension and dual algorithm} \label{ss_ne}

In all this subsection, we consider a matrices graph, defining some continued fraction algorithm on $\R_+^d \times S$.

\begin{define}
    We say that $(D_i)_{i \in S}$ are \defi{domains} of the matrices graph if they are cones of $\R_+^d$ such that
    \[
        \forall i \in S,\ D_i = \biguplus_{j \xrightarrow{m} i} \transp{m} D_j,
    \]
    where $\biguplus$ means that the union is disjoint in Lebesgue measure.
\end{define}


Let $(D_i)_{i \in S}$ be domains for the matrices graph,
and let $D = \bigcup_{i \in S} D_i \times \{i\}$.
Then, the \defi{natural extension} of the continued fraction algorithm is
the map $\tilde{F} : \R_+^{d} \times D \to \R_+^{d} \times D$ defined by
\[
    \tilde{F}(x,(y,i)) = (m^{-1}x, (\transp{m} y, j))
\]
where $i \xrightarrow{m} j$ is a transition such that $x \in m\R_+^{d}$. 
This map is well-defined almost everywhere and preserves the Lebesgue measure of $\R^{d} \times \R^{d} \times S$.
Moreover, by definition of the set $D$, this map is almost everywhere one-to-one.

We define the \defi{dual algorithm} as the map $F^* : D \to D$ such that
\[
    F^*(y,j) = (\transp{m}^{-1}y, i),
\]
if $i \xrightarrow{m} j$ is such that $y \in \transp{m}D_i$.
This dual algorithm is well-defined almost everywhere on $D$.
However, the set $D$ could have zero Lebesgue measure.
For example, the dual of the fully subtractive algorithm is the Arnoux-Rauzy algorithm,
defined on a set $D$ of zero Lebesgue measure called the Rauzy gasket (see~\cite{AHS} for more details).

For each state $i$, we define the \defi{domain language} $\D_i$ as the set of words $m_1^t ... m_n^t$ such that
$m_n...m_1$ labels a path toward $i$ in the matrices graph. 
It is a rational language.
Notice that languages $\D_i$ are stable under prefixes, thus
\[
    \Lambda_{\D_i} = \{y \in \R_+^d \mid \forall n \in \N,\ \exists i_n \xrightarrow{m_n} ... \xrightarrow{m_1} i_0=i,\ y \in \transp{m_1} ... \transp{m_n} \R_+^d \}.
\]

It gives an upper bound of domains.
\begin{lemme}
    If $(D_i)_{i \in S}$ are domains, then $\forall i \in S$, $D_i \subseteq \Lambda_{\D_i}$.
\end{lemme}
And in some cases, $\Lambda_{D_i}$ are domains.
It is the case, for example, with the fully subtractive on two letters, and it is what we use in Section~\ref{sec:wl2} and in Section~\ref{sec:quadratic}.
Domain languages are also useful to find domains that are unions of simplices, see Section~\ref{sec:algo:density}.

\subsection{Invariant densities from natural extension}

In the following, we consider some matrices graph with vertices $S$,
and a set $D = \bigcup_{i \in S} D \times \{i\}$ for some domains $(D_i)_{i \in S}$.
Let $\Gamma$ be the subset of $\R_+^{d} \times D$ defined by
\[
    \Gamma = \{(x,(y,i)) \in \R_+^{d} \times D \mid (y|x) \leq 1 \}.
\]
Notice that the natural extension $\tilde{F}$ preserves this set $\Gamma$.
Hence, we have
\begin{lemme}
    The maps $f_i(x) = \lambda(\{ y \in D_i \mid (y|x) \leq 1\})$ define an invariant density
    for the continued fraction algorithm.
\end{lemme}

Such map $f_i$ can be explicitly computed if the domain $D_i$ is a simplex, thanks to the following formulae due to Veech, see~5.4 in~\cite{Veech78}.

\begin{prop}[Veech] \label{lsd}
    Let $m$ be a square matrix of size $d$.
    Then we have
    \[
        \lambda( \{y \in m \R_+^{d} \mid (y | x) \leq 1 \} ) = \frac{\abs{\det(m)}}{d(d!) \prod (\transp{m} x)},
    \]
    where $\prod (\transp{m} x)$ is the product of coefficients of the vector $\transp{m} x$. 
\end{prop}

Hence, if the domain $D_i$ is a finite union of Lebesgue-disjoint simplices, then we can compute the density, and it is a rational fraction.

If the domain $D_i$ is not a finite union of simplices, we don't have an explicit formulae,
but we can write the density as an integral:
\begin{prop} \label{prop:int}
    For every measurable cone $E \subseteq \R_+^d$ and every $x \in \R_+^d$,
    \[
        \lambda(\{ y \in E \mid (y|x) \leq 1\}) = \frac{1}{d} \int_{E \cap \Delta} \frac{d\nu(y)}{(y | x)^d},
    \]
    where $\Delta = \{y \in \R_+^d \mid \abs{y}_1 = 1\}$ and $\nu$ is the Lebesgue measure on $\Delta$.
\end{prop}

\begin{proof}
    Let $x \in \R_+^d$, and let $A \in M_d(\R_+)$ be an invertible matrix such that $\transp{A} (1, \dots, 1) = x$,
    giving $\abs{A y}_1 = (y \mid x)$.
    Following Veech's notation, let $\mathcal{L}_A: \Delta \to \Delta$ be the map defined by $\mathcal{L}_A(y) = \frac{Ay}{\abs{Ay}_1}$.
    By Proposition~5.2 in~\cite{Veech78}, the Jacobian determinant of $\mathcal{L}_A$ is $Jac_A(y) = \frac{\abs{\det(A)}}{\abs{A y}_1^d}$. Thus we have
    \begin{eqnarray*}
        \abs{\det(A)} \int_{E \cap \Delta} \frac{d\nu(y)}{(y | x)^d} &=& \int_{E \cap \Delta} Jac_A(y) d\nu(y) \\
            &=& \nu(\mathcal{L}_A(E \cap \Delta)) \\
            &=& d \cdot \lambda(\{y \in AE \mid \abs{y}_1 \leq 1\}) \\
            &=& d \cdot \abs{\det(A)} \lambda(\{y \in E \mid (y \mid x) \leq 1\}). \quad \quad \quad \qedhere
    \end{eqnarray*}
\end{proof}

Note that in this proposition, the set of representatives $\Delta$ of elements of $\R_+^2/\R_+^*$
can be replaced by any choice of representatives.

Thanks to this proposition, we see that if there exists domains, then the corresponding density is very regular (infinitely differentiable).
However, if $\lambda(D_i) = 0$, the corresponding density is the null function.
Hence, if every domain has zero Lebesgue measure, it does not permit obtaining a non-trivial invariant density.

\section{Computation of an invariant density for a matrices graph} \label{sec:algo:density}

In this section, we compute domains for a matrices graph.
We do it by describing the continued fraction algorithm as an extension of a simpler one, in Subsection~\ref{ss:minimize}.
Then, in Subsection~\ref{ss:cp}, we compute domains that are convex polyhedra, for the simpler algorithm.
We compute domains for the original algorithm as union of simplices, in Subsection~\ref{ss:ext}.

\subsection{Minimized continued fraction algorithm} \label{ss:minimize}

The first step of the algorithm is to minimize the continued fraction algorithm.
This is done by seeing the matrices graph as a deterministic automaton,
and computing the minimal automaton. That is the unique deterministic automaton recognizing the same language with the minimal number of states. It can be done for example using the Hopcroft's minimization algorithm.

If the continued fraction algorithm was an extension of a simpler algorithm,
the minimization permits to work with the simpler algorithm.

\begin{rem} \label{rem:min}
    The minimal automaton can be defined as a quotient of the original automaton, by the equivalence relation on states defined by $i \sim j \Longleftrightarrow L_i = L_j$, where $L_i$ is the language of state $i$, that is the language of the automaton where the initial state is changed to $i$.
    Hence, the minimized algorithm can be seen as a matrices graph with states that are set of states of the original matrices graph. 
\end{rem}

\begin{ex}
    Consider the following matrices graph. 
    \begin{center}
        \begin{tikzpicture}[node distance=5cm,->,-stealth]
            \node[state] (q0) {$0$};
            \node[state, right of=q0] (q1) {$1$};
            
            \draw (q0) edge[bend left] node[above]{$\left(\begin{array}{rrr}
                    1 & 1 & 0 \\
                    0 & 0 & 1 \\
                    0 & 1 & 0
                    \end{array}\right)$} (q1)
                (q1) edge[loop above] node[above]{$\left(\begin{array}{rrr}
                    1 & 1 & 0 \\
                    0 & 0 & 1 \\
                    0 & 1 & 0
                    \end{array}\right)$} (q1)
                (q1) edge[bend left] node[above=1] {$\left(\begin{array}{rrr}
                    0 & 1 & 0 \\
                    1 & 0 & 0 \\
                    0 & 1 & 1
                    \end{array}\right)$} (q0)
                (q0) edge[loop above] node[above] {$\left(\begin{array}{rrr}
                    0 & 1 & 0 \\
                    1 & 0 & 0 \\
                    0 & 1 & 1
                    \end{array}\right)$} (q0);
        \end{tikzpicture}
    \end{center}
    Then, the minimized matrices graph is the Cassaigne algorithm (see Subsection~\ref{ss:cassaigne}) since the language of both states are all finite words
    over the alphabet
    \[
        \left\{
        \left(\begin{array}{rrr}
        0 & 1 & 0 \\
        1 & 0 & 0 \\
        0 & 1 & 1
        \end{array}\right), \left(\begin{array}{rrr}
        1 & 1 & 0 \\
        0 & 0 & 1 \\
        0 & 1 & 0
        \end{array}\right)\right\}.
    \]
\end{ex}

\subsection{Computation of convex polyhedra domains} \label{ss:cp}

In this subsection, we give an algorithm to compute convex polyhedra that are domains,
for a given matrices graphs with set of states $S$.

We say that a cone $C \subseteq \R_+^d$ is a \defi{convex polyhedron} if
there exists a matrix $m \in M_{d,k}(\R_+)$ such that $C = m \R_+^k$.
We say that a vector $v \in C$ is an \defi{extremal point} of a cone $C \subseteq \R_+^d$
if $C \backslash \R_+ v$ is convex.

The algorithm to compute convex polyhedra is based on the following.
\begin{lemme}
    Assume that $(D_i)_{i \in S}$ are domains with finitely many extremal points of sum $1$.
    Then, for every $i \in S$ and every extremal point $x$ of $D_i$, we have
    \[
        x = \transp{m_1} \ldots \transp{m_n} v 
    \]
    where $j \xrightarrow{m_n} ... \xrightarrow{m_1} i$ is a path in the matrices graph,
    and $v$ is a positive left eigenvector of the product of matrices along a simple loop starting and ending at $j$.
\end{lemme}

\begin{proof}
    Let $i_0 \in S$, and consider an extremal point $x_0$ of $D_{i_0}$.
    Since $D_{i_0}$ is the union of cones $\transp{m_1} D_{i_1}$ over all transitions $i_1 \xrightarrow{m_1} i_0$, there must exist an extremal point $x_1$ of $D_{i_1}$ such that $x_0 = \transp{m_1} x_1$, corresponding to a transition $i_1 \xrightarrow{m_1} i_0$.
    Repeating this process, we obtain a sequence $(x_n)_{n \in \N}$ of extremal points and an infinite path $... i_n \xrightarrow{m_n} ... \xrightarrow{m_1} i_0$ in the matrices graph such that for every $n \in \N$, $x_0 = \transp{m_1} ... \transp{m_n} x_n$.
    Moreover, since the number of extremal points of sum $1$ is finite, this path can be chosen ultimately periodic: if we encounter a vertex already visited, we can iterate the loop indefinitely.
    We then obtain an extremal point $x_n$
    that is a left eigenvector of the product of matrices along the loop.
\end{proof}

In fact we compute a more restrictive set than the set of positive left eigenvectors: if we have a loop at vertex $i$ with matrix $m$ in the matrices graph, and if $\lambda(D_i) \neq 0$, then $D_i$ contains a \defi{limit point}, that is an element of $\bigcap_{n \to \infty} \transp{m}^n (\R_+^*)^d$.
Notice that the set of limit points of sum $1$ of simple loops can be infinite in general.
In that case we chose extremal points of such sets in order to consider finitely many vectors, but we are not sure to find convex polyhedra domains even if there exist.

The idea of the algorithm is to compute such limit points $v$ for simple loops of the matrices graph, and then to stabilize it by adding vectors for each edge. It gives some convex cones $D_i$, and the stabilization step guarantees if it terminates that $\bigcup_{i \xrightarrow{m} j} \transp{m} D_i \subseteq D_j$.

More precisely, the algorithm is as follow:
\begin{enumerate}
    \item For each state $i$ of the matrices graph, compute matrices $m$ of simple loops starting at $i$,
    and for each such matrix $m$, compute extremal points of sum $1$ of $\lim_{n \to \infty} \transp{m}^n (\R_+^*)^d$
    using the Jordan form.
    
    \item For each transition $i \xrightarrow{m} j$ in the matrices graph,
    and for each vector $V$ computed for state $i$, we add vector $\transp{m} V$ to state $j$.
    Then, for each state $i$, we keep only extremal points: if a vector is a linear span
    of other vectors with non-negative coefficients, then we remove it. \label{step}
    
    \item Re-do Step~\ref{step} until the number of vectors doesn't increase.
    
    \item For each state $i \in S$, form a matrix $m_i$ with vectors found as columns. Check that
    $\forall j \in S$, $m_j \R_+^{k_j} = \biguplus_{i \xrightarrow{m} j} \transp{m} m_i \R_+^{k_i}$.
    Then return the domains $D_j = m_j \R_+^{k_j}$ found.
\end{enumerate}

\begin{ex}
    Consider the Cassaigne's continued fraction algorithm (see Subsection~\ref{ss:cassaigne}).
    The two only simple loops are the two trivial ones, labeled respectively by
    matrices $m_0 =
    \left(\begin{array}{rrr}
    0 & 1 & 0 \\
    1 & 0 & 0 \\
    0 & 1 & 1
    \end{array}\right)$ and $m_1 = \left(\begin{array}{rrr}
    1 & 1 & 0 \\
    0 & 0 & 1 \\
    0 & 1 & 0
    \end{array}\right)$.
    It gives corresponding left eigenvectors $(1,1,0)$ and $(0,1,1)$.
    Then, if we consider images of these vectors by the transposed of the matrices, we get one more vector $(1,1,1)$.
    
    Then, we easily check that the cone $C = \left(\begin{array}{rrr}
    0 & 1 & 1 \\
    1 & 1 & 1 \\
    1 & 0 & 1
    \end{array}\right) \R_+^{d}$ satisfies
    \[
        \transp{m_0} C \uplus \transp{m_1} C = C.
    \]
    Thus the algorithm terminates here,
    and $C$ is a domain.
\end{ex}

\subsection{Computation of domains for an extension} \label{ss:ext}

We assume that domains $(D_I')_{I \in S'}$ have been found for the minimized algorithm,
and we want to compute domains $(D_i)_{i \in S}$ for the original algorithm,
which is an extension of the minimized one.
By Remark~\ref{rem:min}, we can assume that $\forall I \in S',\ I \subset S$.
Then, we have the following.

\begin{lemme} \label{lem:Ii}
    $\forall I \in S',\ \D_I' = \bigcup_{i \in I} \D_i$.
\end{lemme}

\begin{proof}
    By definition of the minimized matrices graph, for every $I,J \in S'$ 
    \[
        I \xrightarrow{m} J \quad \Longleftrightarrow \quad \exists i \in I,\ \exists j \in J,\ i \xrightarrow{m} j
                            \quad \Longleftrightarrow \quad \forall i \in I,\ \exists j \in J,\ i \xrightarrow{m} j.
    \]
    We easily deduce the equality: a word of $\D_I'$ is a label of a path in the minimized matrices graph, and it corresponds to a path in the original matrices graph toward some $i \in I$, whose label is a word of $\D_i$.
\end{proof}

In particular, every language $\D_i$ is included in some language $\D_I'$.
Since we can iterate the decomposition $\D_I' = \pref(\bigcup_{J \xrightarrow{\transp{m}} I} m \D_J')$, it is natural to decompose each $\D_i$ as:
\[
    \D_i = \pref( \bigcup_{J \in S'} W_{i,J} \D_J' ),
\]
for some languages $W_{i,J}$ of words of $\D_i$ labeling paths from $i$ to states $j \in J$.

If such finite union of languages can be found, then we define
\[
    D_i := \bigcup_{J \in S'} W_{i,J} D_J',
\]
for every $i \in S$, where we denote $W_{i,J} D_J' := \bigcup_{u_1...u_k \in W_{i,J}} u_1...u_k D_J'$.

\begin{prop} \label{prop:ext}
    $(D_i)_{i \in S}$ are domains.
\end{prop}

\begin{proof}
    Let $i \in S$. Let us show that $D_i = \biguplus_{j \xrightarrow{\transp{m}} i} m D_j$.
    
    Since $\D_i = \pref(\bigcup_{j \xrightarrow{\transp{m}} i} m \D_j)$, we have
    \[
        \pref(\bigcup_{j \xrightarrow{\transp{m}} i} m \bigcup_{K \in S'} W_{j,K} \D_K') = \pref(\bigcup_{J \in S'} W_{i,J} \D_J') = \D_i \subseteq \D_I'
    \]
    By definition of the minimized graph, each transition $j \xrightarrow{\transp{m}} i$ corresponds to a transition $J \xrightarrow{\transp{m}} I$ in the minimized graph, with $j \in J$ and $i \in I$.
    Thus, by iterating the decomposition $D_I' = \biguplus_{J \xrightarrow{\transp{m}} I} m D_J'$, we get
    \[
        \biguplus_{j \xrightarrow{\transp{m}} i} m D_j
        = \biguplus_{j \xrightarrow{\transp{m}} i} m \bigcup_{K \in S'} W_{j,K} D_K'
        = \bigcup_{J \in S'} W_{i,J} D_J' = D_i. 
    \]
\end{proof}

Such a decomposition of rational languages can be obtained with usual operations on rational languages.
Indeed, for each $i \in S$, take the minimal automaton recognizing $\D_i$.
For each $J \in S'$, compute the set $S_J$ of states of this minimal automaton whose language is $\D_J'$ (we have $\abs{S_J} \leq 1$).

Then remove every outgoing edge from each state in $S_J$, for each $J \in S'$.
Then, let $W_{i,J}$ be the language recognized by the automaton where
we set $S_J$ as the set of final states.
Then, we easily check that $\D_i = \pref(\bigcup_{J \in S'} W_{i,J} \D_J')$ as soon as
the set of states $\bigcup_{J \in S'} S_J$ is reachable from every state.

Then, 
the formulae of Proposition~\ref{lsd} permits to compute the invariant density, by splitting this union as Lebesgue-disjoint simplices.

\begin{ex}
    Consider the following extension of the Cassaigne continued fraction algorithm.
    \begin{center}
        \begin{tikzpicture}[node distance=5cm,->,-stealth]
            \node[state] (q0) {$0$};
            \node[state, right of=q0] (q1) {$1$};
            
            \draw (q0) edge[bend left] node[above]{$\left(\begin{array}{rrr}
                    1 & 1 & 0 \\
                    0 & 0 & 1 \\
                    0 & 1 & 0
                    \end{array}\right)$} (q1)
                (q1) edge[loop above] node[above]{$\left(\begin{array}{rrr}
                    1 & 1 & 0 \\
                    0 & 0 & 1 \\
                    0 & 1 & 0
                    \end{array}\right)$} (q1)
                (q1) edge[bend left] node[above=1] {$\left(\begin{array}{rrr}
                    0 & 1 & 0 \\
                    1 & 0 & 0 \\
                    0 & 1 & 1
                    \end{array}\right)$} (q0)
                (q0) edge[loop above] node[above] {$\left(\begin{array}{rrr}
                    0 & 1 & 0 \\
                    1 & 0 & 0 \\
                    0 & 1 & 1
                    \end{array}\right)$} (q0);
        \end{tikzpicture}
    \end{center}
    
    The minimized matrices graph is the Cassaigne algorithm whose a domain is
    $D_0' = \left(\begin{array}{rrr}
    0 & 1 & 1 \\
    1 & 1 & 1 \\
    1 & 0 & 1
    \end{array}\right) \R_+^{d}$, and whose domain language is $\D_0' = \{\transp{m_0}, \transp{m_1}\}^*$,
    where $m_0 =
    \left(\begin{array}{rrr}
    0 & 1 & 0 \\
    1 & 0 & 0 \\
    0 & 1 & 1
    \end{array}\right)$ and $m_1 = \left(\begin{array}{rrr}
    1 & 1 & 0 \\
    0 & 0 & 1 \\
    0 & 1 & 0
    \end{array}\right)$.
        
    We easily see that $\D_0$ is the language of all finite words over $\{\transp{m_0}, \transp{m_1}\}$
    starting by $\transp{m_0}$, and $\D_1$ is the language of all words over the same alphabet
    starting by $\transp{m_1}$.
    
    Thus, we have $\D_0 = \transp{m_0} \D_0'$ and $\D_1 = \transp{m_1} \D_0'$, so we get that $D_0 = \transp{m_0} D_0' = \left(\begin{array}{rrr}
    1 & 1 & 1 \\
    1 & 1 & 2 \\
    1 & 0 & 1
    \end{array}\right) \R_+^{d}$ and $D_1 = \transp{m_1} D_0' = \left(\begin{array}{rrr}
    0 & 1 & 1 \\
    1 & 1 & 2 \\
    1 & 1 & 1
    \end{array}\right) \R_+^{d}$ are domains.
    
    Then, we get an invariant density thanks to the formulae of Proposition~\ref{lsd}:
    $f_0(x,y,z) = \frac{1}{(x+y)(x+y+z)(x+2y+z)}$ and $f_1(x,y,z) = \frac{1}{(y+z)(x+y+z)(x+2y+z)}$.
\end{ex}

\section{Converting a continued fraction algorithm to a matrices graph} \label{sec:general_to_matrices}

In this section, we give an algorithm to convert a general continued fraction algorithm defined as a piecewise linear map on pieces that are polyhedra,
to a matrices graph (see subsection~\ref{ss_cfa} for a definition).
The algorithm may not terminate, but it terminates for every classical continued fraction algorithm.

We represent the input continued fraction algorithm as a graph with vertices $S$ and edges
$i \xrightarrow{m,D} j$, where $m \in M_d(\R_+)$ and $D$ is a non-negative matrix with $d$ rows such that $m^{-1} D$ is non-negative.
We denote $c(D)$ the number of columns of $D$.
We assume that for every $i \in S$, the union
$\bigcup_{i \xrightarrow{m,D} j} D \R_+^{c(D)}$ is Lebesgue-disjoint.

It defines a partial map $F:\R_+^d \times S \to \R_+^d \times S$ by
$F(x,i) = (m^{-1} x, j)$ if $i \xrightarrow{m,D} j$ such that $x \in D \R_+^{c(D)}$.
The map is defined almost everywhere if for every $i \in S$, $\bigcup_{i \xrightarrow{m,D} j} D \R_+^{c(D)} = \R_+^d$,
but we don't need this hypothesis in the following.

The algorithm outputs a graph labeled by square matrices, with states $S \times \M$, where $\M$ is a finite set of square matrices.
See Algorithm~\ref{algo_to_mat} for the algorithm.

\begin{algorithm}
    Start with an empty result graph\;
    $d \gets $ number of rows of matrices\;
    $\mathcal{S} \gets \{(i, I_d) \mid i \in S \}$\;
    Set $(i,I_d)$ as seen\;
    \While{$\mathcal{S} \neq \emptyset$}
    {
        Withdraw $(i,I)$ from $\mathcal{S}$\;
        \For{$i \xrightarrow{m,D} j$}
        {
            \If{$\dim(D \R_+^{c(D)} \cap I \R_+^{d}) = d$}
            {
                Compute a set $\J$ of square matrices such that
                $\biguplus_{J \in \J} J \R_+^{d} = m^{-1} (D \R_+^{c(D)} \cap I \R_+^{d})$\;
                \For{$J \in \J$}
                {
                    Add the edge $(i,I) \xrightarrow{I^{-1}mJ} (j,J)$ to the result graph\;
                    \If{$(j,J)$ has not been seen}
                    {
                        Add $(j,J)$ to $\mathcal{S}$\;
                        Set $(j,J)$ as seen\;
                    }
                }
            }
        }
    }
    \caption{Convert a continued fraction algorithm to a matrices graph} \label{algo_to_mat}
\end{algorithm}

\begin{rem}
    In Algorithm~\ref{algo_to_mat} there is choices to make, to partition the cone $m^{-1} (D \R_+^{c(D)} \cap I \R_+^{d})$ by projective simplices $J \R_+^d$. 
    And we have to normalize such matrices $J$ in order that same cones give same matrices.
\end{rem}

The following proposition explain the link between the input continued fraction algorithm
and the output matrices graph. It shows that under a small hypothesis on the input graph, the output
graph is indeed almost a matrices graph.

\begin{prop} \label{prop:gen:to:mat}
    Let $F : \R_+^{d} \times S \to \R_+^{d} \times S$ be the continued fraction algorithm corresponding to the input.
    If Algorithm~\ref{algo_to_mat} terminates, then the output continued fraction algorithm
    $G: \R_+^{d} \times S \times \M \to \R_+^{d} \times S \times \M$ is well-defined, and
    for every $n \in \N$, 
    \[
        F^n \circ \phi = \phi \circ G^n,
    \]
    where $\phi : \R_+^{d} \times S \times \M \to \R_+^{d} \times S$ is defined by $\phi((x,i,I)) = (Ix, i)$.
    
    Moreover, if the input graph satisfies
    \[
        \forall i \xrightarrow{m,D} j, \quad m^{-1} D \R_+^{c(D)} \subseteq \bigcup_{j \xrightarrow{m',D'} k} D' \R_+^{c(D')},
    \]
    then the graph obtained by the Algorithm~\ref{algo_to_mat} is a matrices graph, up to removing vertices with the identity matrix.
\end{prop}

\begin{proof}
    Throughout the proof, let $(i,I) \in S \times \M$ be a state of the output graph.
    Let us prove that the union $\bigcup_{(i,I) \xrightarrow{M} (j,J)} M \R_+^d$ is Lebesgue-disjoint.
    Let $(i,I) \xrightarrow{M} (j,J)$ and $(i,I) \xrightarrow{M'} (k,K)$ be two distinct edges.
    Then, there exists $i \xrightarrow{m,D} j$ and $i \xrightarrow{m',D'} k$ such that
    $M = I^{-1} m J$ and $M' = I^{-1} m' K$.
    If these edges are identical, then $J \R_+^d$ and $K \R_+^d$ are Lebesgue-disjoint by construction,
    thus $M \R_+^d$ and $M' \R_+^d$ are Lebesgue-disjoint.
    Otherwise, by the hypothesis $D \R_+^{c(D)}$ and $D' \R_+^{c(D')}$ are Lebesgue-disjoint,
    and $J \R_+^d \subseteq m^{-1} D \R_+^{c(D)}$ and $K \R_+^d \subseteq m'^{-1} D' \R_+^{c(D')}$ by construction, thus $M \R_+^d$ and $M' \R_+^d$ are Lebesgue-disjoint.
    
    Hence, the map $G$ is well-defined, but not necessarily on the whole positive cone.
    Let us show that we have $\phi(G(x,i,I)) = F(I x,i)$ for all $(i,I) \in S \times \M$ and for almost every $x \in \R_+^d$.
    If we have a transition $(i,I) \xrightarrow{M} (j,J)$ in the output graph,
    then we have a transition $i \xrightarrow{m, D} j$ in the input graph, with $M = I^{-1} m J$
    and $J \R_+^d \subseteq m^{-1} (I \R_+^d \cap D \R_+^{c(D)})$.
    Thus, if $G(x, i, I) = (M^{-1} x, j, J)$, then we have $x \in M \R_+^d$,
    so $x \in m J \R_+^d \subseteq D \R_+^{c(D)}$.
    Hence, $F(I x,i) = (m^{-1} I x, j) = \phi((J^{-1}m^{-1} I x, j, J)) = \phi(G(x, i, I))$.
    Then, we get the equality $F^n \circ \phi = \phi \circ G^n$ for every $n \in \N$ by iterating.
    
    By construction, for every $i \xrightarrow{m,D} j$, there exists a set $\J_j$ such that
    \[
        \bigcup_{J \in \J_j} J \R_+^d = m^{-1}(I \R_+^d \cap D \R_+^{c(D)}),
    \]
    thus
    \begin{eqnarray*}
        \bigcup_{(i,I) \xrightarrow{M} (j,J)} M \R_+^d
        &=& \bigcup_{i \xrightarrow{m,D} j} \bigcup_{J \in \J_j} I^{-1}mJ\R_+^d \\
        &=& \bigcup_{i \xrightarrow{m,D} j} I^{-1} m m^{-1} (I \R_+^d \cap D \R_+^{c(D)}) \\
        &=& \R_+^d \cap I^{-1} \bigcup_{i \xrightarrow{m,D} j} D \R_+^{c(D)}.
    \end{eqnarray*}
    If moreover we assume the additional hypothesis on the input graph and that $I \neq I_d$,
    then $\R_+^d \cap I^{-1} \bigcup_{i \xrightarrow{m,D} j} D \R_+^{c(D)} = \R_+^d$ since
    by construction $I \R_+^d$ is included in
    $m'^{-1}(K\R_+^d \cap D' \R_+^{c(D')}) \subseteq \bigcup_{i \xrightarrow{m,D} j} D \R_+^{c(D)}$
    for some transition $k \xrightarrow{m',D'} i$ in the input graph, and some $K$.
    Thus, the output graph is a matrices graph up to remove states with the identity matrix.
\end{proof}

We can recover the invariant density of the original continued fraction algorithm
from the invariant density of the computed matrices graph thank to the following formulae.

\begin{lemme} \label{lem:general:density}
    Let a continued fraction algorithm be defined by a graph with set of vertices $S$,
    and let a corresponding matrices graph with vertices $S \times \M$, obtained from
    Algorithm~\ref{algo_to_mat}. Let $(f_{i,I})_{(i,I) \in S \times \M}$ be an invariant density
    for the matrices graph. Then, for every $i \in S$ the maps
    \[
        f_i(x) = \sum_{I \in \M,\ x \in I \R_+^d} \abs{\det(I^{-1})} f_{i,I}(I^{-1}x)
    \]
    define an invariant density for the input continued fraction algorithm.
\end{lemme}

\begin{proof}
    Let $\mu$ be the measure on $\R_+^d \times S$ with density $(f_i)_{i \in S}$ with respect to Lebesgue,
    and let $\nu$ be the measure on $\R_+^d \times S \times \M$ with density $(f_{i,I})_{(i,I) \in S \times \M}$ with respect to Lebesgue.
    We easily check that $\mu = \phi_{*} \nu$, where $\phi : \R_+^d \times S \times \M \to \R_+^d \times S$ is defined by $\phi(x,i,I) = (I x, i)$.
    Then, we check that $\mu$ is an invariant density, using the fact that $\nu$ is an invariant density and using the equality $F \circ \phi = \phi \circ G$ given by Proposition~\ref{prop:gen:to:mat}.
\end{proof}

Notice that the invariant density for a general continued fraction algorithm is not
continuous in general, since the condition $x \in I \R_+^d$ is not continuous in general.
An example is given in Subsection~\ref{ss:brun}.

\section{Win-lose graph on two letters} \label{sec:wl2}

In Section~\ref{sec:algo:density}, we gave an algorithm to compute the invariant density for some matrices graphs.
But the algorithm may fail to find the invariant density, although it is formed of rational fractions. 
In this section, we give an algorithm that decides whether the invariant density is composed of rational fractions and computes such rational fractions, for every
continued fraction algorithm given by a win-lose graph on two letters.
It is done by describing domains, by computing their boundaries.

\subsection{Computation of boundaries}

In this subsection, we give a way to compute the boundary of the limit set of some rational language $L$ stable by prefixes, over the alphabet $\{\0,\1\}$, where we denote
$\0 = \begin{pmatrix} 1 & 1 \\ 0 & 1 \end{pmatrix}$
and $\1 = \begin{pmatrix} 1 & 0 \\ 1 & 1 \end{pmatrix}$ to lighten the notations.

In order to describe the boundary of the limit set, we need to understand which
words correspond to neighboring cones.
It is given by the following.

\begin{lemme}
    For every $n \in \N$ and every $(u,v) \in (\{\0,\1\}^2)^n$, we have
    \[
        u \R_+^d \cap v \R_+^d \neq \emptyset
    \]
    if and only if $(u,v)$ is recognized by the automaton
    \begin{center}
        \begin{tikzpicture}[node distance=2cm,-stealth]
            \node[state, double, thick] (q0) {};
            \node[state, double, right of=q0] (q1) {};
            \node[state, double, left of=q0] (q2) {};
            
            \draw (q0) edge[loop above] node[above]{$(\0,\0)$} (q0)
                (q0) edge[loop below] node[below]{$(\1,\1)$} (q0)
                (q0) edge node[above] {$(\0,\1)$} (q1)
                (q0) edge node[above] {$(\1,\0)$} (q2)
                (q1) edge[loop above] node[above]{$(\1,\0)$} (q1)
                (q2) edge[loop above] node[above]{$(\0,\1)$} (q2);
        \end{tikzpicture}
    \end{center}
    where the central state is initial, and every state is final.
\end{lemme}

\begin{proof}
    Easy verification. 
\end{proof}

We call the automaton of this lemma the \defi{relations automaton},
and its language is denoted by $\Lrel$.
It permits to compute the boundary:

\begin{prop} \label{prop:boundary}
    Let $L \subseteq \{\0,\1\}^*$ be a rational language stable by prefixes.
    The boundary of the limit set of $L$
    is the limit set of the rational language
    \[
        \partial L := \pruneinf(p_1(L \times \pref(L^c) \cap \Lrel)) \cup L^{min} \cup L^{max},
    \]
    where $L^c = \{\0,\1\}^* \backslash L$ denotes the complementary of $L$,
    $L^{min}$ and $L^{max}$ are respectively the languages recognizing the
    smallest and the greatest words of $L$ in lexicographical order,
    $p_1 : (\{\0,\1\}^2)^* \to \{\0,\1\}$ is the projection on the first coordinate,
    and $\pruneinf(L)$ remove all words of $L$ that cannot be extended to an arbitrarily longer word of $L$.
\end{prop}

\begin{proof}
    Language $L^{min}$ (respectively $L^{max}$) is rational and is easily computed
    by following the minimal (respectively maximal) path in the automaton of $L$ up to a loop.
    Then, the language $\partial L$ is rational since it is obtained
    by usual operations on rational languages
    (complementary, prefixes, intersection, union, image by the morphism $p_1$).
    The $\pruneinf$ also preserves the fact to be rational: indeed, if a pruned automaton recognizes
    a language $L'$, then $\pruneinf(L')$ is recognized by the same automaton where we remove every state
    from which we cannot reach a loop.
    The operation $\pruneinf$ doesn't change the limit set of the language, but it permits getting a simpler language.
    
    Note that $\partial L$ is stable under prefixes, as this property is preserved by product, intersection, union,
    projection on the first coordinate, and $\pruneinf$.
    
    Let us show that $\partial \Lambda_L = \Lambda_{\partial L}$.
    The boundary $\partial \Lambda_L$ is equal to
    $(\overline{\R_+^{d} \backslash \Lambda_L} \cap \Lambda_L) \cup \Lambda_{L^{min}} \cup \Lambda_{L^{max}}$,
    since the lexicographical order corresponds to the order on $(\R_+^*)^2/\R_+^*$ defined by
    $(x:y) \leq (x':y')$ if and only $x'y \leq xy'$.
    
    Suppose $x \in \overline{\R_+^{d} \backslash \Lambda_L} \cap \Lambda_L$.
    Then, there exists a sequence of elements $x_n \in \R_+^{d} \backslash \Lambda_L$,
    such that $\lim_{n \to \infty} x_n = x$.
    Let $n \in \N$. As $x \in \Lambda_L$, there exists $u_1...u_n \in L$ such that
    $x \in u_1...u_n \R_+^d$.
    Let $N \in \N$ be large enough such that $x_N \in v_1...v_{n} \R_+^d$, for some
    word $v_1...v_n \in \{\0,\1\}^n$ such that $(u_1,v_1)...(u_n,v_n) \in \Lrel$.
    As $x_N \not \in \Lambda_L$, there exists an extension of the word $v_1...v_n$ that belongs
    to $L^c$. Thus, we have $u \in p_1(L \times \pref(L^c) \cap \Lrel)$, and we conclude that $x \in \Lambda_{\partial L}$.
    
    Conversely, suppose $x \in \Lambda_{p_1(L \times \pref(L^c) \cap \Lrel)}$.
    Then, for every $n \in \N$, there exists $u_1...u_n \in L$ and $v_1...v_n \in \pref(L^c)$ such that
    $(u_1,v_1)...(u_n,v_n) \in \Lrel$ and $x \in u_1...u_n \R_+^d$.
    In particular, $x \in \Lambda_L$.
    And for every $N \in \N$ such that $v_1...v_N \in L^c$, the interior of $v_1...v_N \R_+^d$ is disjoint of $\Lambda_L$,
    and the distance between $v_1...v_N \R_+^d$ and $x$ is less than $2 \norm{x}_1/n$.
    Thus, there exists a sequence $(x_n)_{n \in \N} \in (\R_+^d \backslash \Lambda_L)^\N$ such that
    $x = \lim_{n \to \infty} x_n$, so $x \in \partial \Lambda_L$.
\end{proof}

\subsection{Decomposition as union of intervals}

We show in this subsection that domains can be decomposed as countable union of intervals, union a subset of zero Lebesgue measure.

\begin{lemme} \label{lem:AB}
    Let $L$ be a rational language over the alphabet $\{\0, \1\}$.
    Then, there exists two rational languages $A$ and $B$ such that
    $L = A \{\0, \1\}^* \cup B$, with $\lambda(\Lambda_B) = 0$ and $\lambda(\Lambda_A) = 0$.
    Moreover, $A$ and $B$ are computable.
\end{lemme}

\begin{proof}
    Consider the minimal automaton that recognizes the language $L$.
    In this automaton, there is at most one state whose language is $\{\0,\1\}^*$.
    If such a state doesn't exists, then take $A = \emptyset$.
    Otherwise, remove every outgoing edge from such a state, and set it as the unique final state.
    Then, the language recognized by this new automaton is $A$.
    We then set $B = L \backslash A \{\0, \1\}^*$.
    Obviously we have $L = A \{\0, \1\}^* \cup B$, and $A$ and $B$ are rational.
    Then, consider the minimal automaton with sink state that recognizes the language $B$.
    It gives a win-lose graph on two letters satisfying the Fougeron's criterion (see~\cite{Fougeron}),
    and the sink state is reachable from every other state, thus $\lambda(\Lambda_B) = 0$.
    The same argument shows that $\lambda(\Lambda_A) = 0$.
\end{proof}

Note that the fully subtractive algorithm for $d=2$ is convergent and auto-dual.
Thus it implies that domains $D_i$ for any extension are unique
and are limit sets of rational languages $\D_i$.
Thus, they are a countable union intervals $\bigcup_{w \in A} m_w \R_+^2$,
up to sets of zero Lebesgue measure of the form $\Lambda_A \cup \Lambda_B$. 

\subsection{Non-rational density}

In this subsection, we prove the following proposition.
With computations from the two previous subsections, it allows us
to algorithmically decide whether the invariant density is composed of rational fractions.

\begin{prop} \label{prop:non-rat}
    Consider a win-lose graph on two letters.
    The unique invariant densities $(f_i)_{i \in S}$ are rational fractions if and only if
    the unique domains $(D_i)_{i \in S}$ are finite unions of intervals, up to sets of zero Lebesgue measure.
\end{prop}

The remaining of this subsection is devoted to the proof of this proposition.


We choose representatives of $\P\R_+^2$ of the form $(x,1)$, where $x \in \R_+ \cup \{\infty\}$.
By Proposition~\ref{prop:int}, $f_i(x) = \frac{1}{d} \int_{y \in p(D_i)} \frac{dy}{(yx+1)^2}$, where $p((x,y)) = (x/y, 1)$.

\begin{lemme} \label{lem:C}
    Let $D \subset \R_+ \cup\{\infty\}$ be the closure of its interior. For $z \in \C$, let
    \[
        f_D(z) := \frac{1}{d}\int_{y \in D} \frac{dy}{(zy + 1)^2}.
    \]
    Then $f_D$ is well defined and holomorphic in $\C \backslash \overline{\{\frac{-1}{y} \mid y \in D\}}$.
    Moreover, $f_D$ can be extended to an holomorphic function on $\C \backslash \{\frac{-1}{y} \mid y \in \partial D\}$,
    and every point of $\{\frac{-1}{y} \mid y \in \partial D\}$ is a singularity of $f_D$.
    In particular, $f_D$ is a rational fraction if and only if $\partial D$ is finite.
\end{lemme}

\begin{proof}
    Let $\epsilon > 0$ and let $U = \{z \in \C \mid d(z, \{\frac{-1}{y} \mid y \in D\}) \geq \epsilon \text{ and } \abs{z} \leq 1/\epsilon\}$.
    Then for $z \in U$, the complex derivative $\frac{-2y}{(zy + 1)^3}$ of $\frac{1}{(zy + 1)^2}$ is dominated by a Lebesgue integrable map on $D$, thus $f_D$ is well-defined and holomorphic on $U$.
    Then $f_D$ is well-defined and holomorphic in $\C \backslash \overline{\{\frac{-1}{y} \mid y \in D\}}$.
    
    Let $y$ be in the interior of $D$.
    Then, $D$ can be decompose as $D = I \cup D'$, where $I$ is an open interval containing $y$ and disjoint of $D'$.
    By the above, $f_{D'}$ is holomorphic at $-1/y$. Furthermore, by Lemma~\ref{lsd}, $f_I$ is a rational fraction with an holomorphic extension at $-1/y$.
    Thus, $f_D = f_{D'} + f_{I}$ has an holomorphic extension at $-1/y$.
    
    Now, we need the following.
    
    {\bf Fact} If $\Gamma$ is a closed path disjoint of $\overline{\{\frac{-1}{y} \mid y \in D\}}$, then $\displaystyle{\oint_\Gamma f_D = 0}$.
    
    Indeed, the map $(y,z) \mapsto \frac{1}{(yz + 1)^2}$ is integrable on $D \times \Gamma$, thus by Fubini's theorem 
    \[
        \oint_{z \in \Gamma} \int_{y \in D} \frac{dy dz}{(yz+1)^2} = \int_{y \in D} \oint_{z \in \Gamma} \frac{dz dy}{(yz+1)^2} = 0,
    \]
    since the residue of $\frac{1}{(yz+1)^2}$ is everywhere zero for every $y \in D$.
    
    Let $y \in \partial D$.
    Let us show that $\frac{-1}{y}$ is a singularity of $f_D$.
    We can find an arbitrarily small closed path $\Gamma$ enclosing $-1/y$ such that
    the path goes through the interior of $\{\frac{-1}{y'} \mid y' \in D\}$ exactly once.
    If we decompose $D = D' \cup I$, where $I$ is an open interval meeting $\Gamma$ and where $D'$ is disjoint from $I$,
    we have $\oint_\Gamma f_{D'} = 0$ by the above fact.
    Furthermore, by Proposition~\ref{lsd}, $f_I$ is a rational fraction with exactly one singularity
    enclosed by $\Gamma$ whose residue is non-zero; thus, $\oint_\Gamma f_I \neq 0$.
    Hence, $\oint_\Gamma f_D = \oint_\Gamma f_{D'} + \oint_\Gamma f_I \neq 0$,
    proving that $\Gamma$ encloses a singularity of $f_D$.
    As $\Gamma$ is arbitrarily close to $-1/y$, this shows that $f_D$ has a singularity at $-1/y$.
\end{proof}

By Lemma~\ref{lem:AB}, we can decompose each domain language $\D_i = A_i \{0,1\}^* \cup B_i$.
Since $\D_i$ is stable under prefixes, $\pref(A_i \{0,1\}^*) \subseteq \D_i$.
Let $L = \pref(A_i \{0,1\}^*)$. The set $\Lambda_{L} = \bigcap_{n \in \N} \bigcup_{w \in L_n} w \R_+^2$ is compact
and is the closure of the countable union of open intervals $\bigcup_{w \in A_i} w (\R_+^*)^2$,
thus it is the closure of its interior.
Furthermore $\lambda(\Lambda_{B_i} \cup \Lambda_{A_i}) = 0$, so
$D_i$ is equal to $\Lambda_L$ union a set of zero Lebesgue measure.
Thanks to Lemma~\ref{lem:C}, the density $f_i = f_{\Lambda_L}$ admits a unique holomorphic extension
with finitely many singularities if and only if
$\Lambda_L$ is a finite union of intervals, if and only if $f_i$ is a rational fraction.
This concludes the proof of Proposition~\ref{prop:non-rat}.

\begin{rem}
    The unique holomorphic extension of a density function can be complicated since it can have a Cantor of singularities: see Example~\ref{ex:cantor}.
\end{rem}

\subsection{The algorithm for two letters win-lose graphs}

In this subsection, we present the algorithm to test if a win-lose graph on two letters has invariant densities that are rational fractions and to compute it if it is the case.
The algorithm is as follows:

\begin{itemize}
    \item Compute domain languages $\D_i$.
    For each $\D_i$, do the following.
    
    \item Decompose $\D_i$ as $\D_i = A \{\0,\1\}^* \cup B$, as in Lemma~\ref{lem:AB}.
    
    \item Compute the language $\partial L$, defined in Proposition~\ref{prop:boundary}, describing the boundary of the limit set of the language
    $L = \pref(A \{0,1\}^*)$.
    
    \item Compute non-trivial strongly connected components of the minimal automaton of $\partial L$.
    
    \item If there exists a component that is not a loop or not terminal, then, by Proposition~\ref{prop:non-rat},
    we know that the density $f_i$ is not a rational fraction since $\Lambda_L$ is an infinite union of disjoint closed intervals.
    
    \item Otherwise, we can decompose $\partial L = \pref( u_1 v_1^* \cup ... \cup u_{2N} v_{2N}^*)$,
    where $v_1$, ..., $v_{2N}$ are labels of loops,
    and $\partial \Lambda_L$ is the set of
    quadratic half lines $\R_+ V_j = \lim_{n \to \infty} m_{u_j} m_{v_j}^n \R_+^2$, where $m_{u_j}$ is the product of matrices of $u_j$.
    Such vectors $V_j$ can be computed by taking a Perron eigenvector of $m_{v_j}$ and multiplying it by $m_{u_j}$.
    Then, we order these vectors for the relation $(x,y) \leq (a,b) \Leftrightarrow ya \leq xb$.
    We get a finite increasing sequence of vectors $V_1$, ..., $V_{2N}$.
    For every $k \in \{1,...,N\}$, let $m_k$ be the matrix with columns $V_{2k-1}$ and $V_{2k}$.
    Then, the domain $D_i$ is equal to the union $m_1 \R_+^2 \cup ... \cup m_N \R_+^2$
    up to a set of Lebesgue measure zero.
    Thus, we deduce that the density at state $i$ is $f_i(x) = \sum_{k=1}^N \frac{\abs{\det(m_k)}}{(V_{2k-1}|x)(V_{2k} | x)}$.
\end{itemize}

\begin{ex}
    Consider the win-lose graph
    \begin{center}
        \begin{tikzpicture}[-stealth,node distance=2cm]
            \node[state] (q0) {};
            \node[state, right of=q0] (q1) {};
            \node[state, left of=q0] (q2) {$0$};
            
            \draw (q0) edge[bend left] node[above] {$1$} (q1)
                (q0) edge[bend left] node[above] {$0$} (q2)
                (q1) edge[bend left] node[above]{$0$} (q0)
                (q1) edge[loop right] node[right]{$1$} (q1)
                (q2) edge[bend left] node[above]{$1$} (q0)
                (q2) edge[loop left] node[left]{$0$} (q2);
        \end{tikzpicture}
    \end{center}
    The domain of the state $0$ is the limit set of the language $\D_0$ of the automaton
    \vspace{-.5cm}
    \begin{center}
        \begin{tikzpicture}[-stealth,node distance=2cm]
            \node[state, double, thick] (q0) {};
            \node[state, double, right of=q0] (q1) {};
            \node[state, double, right of=q1] (q2) {};
            
            \draw (q0) edge[bend left] node[above] {$\0$} (q1)
                (q1) edge[bend left] node[above] {$\1$} (q0)
                (q1) edge node[above]{$\0$} (q2)
                (q2) edge[loop right, out=40, in=-40, looseness=18] node[right]{$\0$} (q2)
                (q2) edge[loop right] node[right]{$\1$} (q2);
        \end{tikzpicture}
    \end{center}
    \vspace{-.7cm}
    The decomposition of Lemma~\ref{lem:AB} gives $A$ such that $\pref(A \{\0,\1\}^*) = \D_0$, thus we compute the language $\partial L$ for $L = \D_0$.
    The language $\pref(L^c)$ is recognized by the automaton
    \vspace{-.5cm}
    \begin{center}
        \begin{tikzpicture}[-stealth,node distance=2cm]
            \node[state, double] (q0) {};
            \node[state, double, right of=q0, thick] (q1) {};
            \node[state, double, right of=q1] (q2) {};
            
            \draw (q0) edge[bend left] node[above] {$\1$} (q1)
                (q1) edge[bend left] node[above] {$\0$} (q0)
                (q1) edge node[above]{$\1$} (q2)
                (q2) edge[loop right, out=40, in=-40, looseness=18] node[right]{$\0$} (q2)
                (q2) edge[loop right] node[right]{$\1$} (q2);
        \end{tikzpicture}
    \end{center}
    \vspace{-.7cm}
    Then, the language $L \times \pref(L^c) \cap \Lrel$ is recognized by
    \begin{center}
        \begin{tikzpicture}[-stealth,node distance=3cm]
            \node[state, double, thick] (q0) {};
            \node[state, double, below right of=q0] (q1) {};
            \node[state, double, above right of=q1] (q2) {};
            \node[state, double, right of=q2] (q3) {};
            
            \draw (q0) edge[bend left] node[above right] {$(\0,\0)$} (q1)
                (q0) edge[bend left] node[above] {$(\0,\1)$} (q2)
                (q1) edge[bend left] node[below left] {$(\1,\1)$} (q0)
                (q1) edge node[below right]{$(\0,\1)$} (q2)
                (q2) edge node[above]{$(\1,\0)$} (q3);
        \end{tikzpicture}
    \end{center}
    If we project on first coordinate, we obtain the language $\pref((\0\1)^*\0\0\1)$.
    Then, after $\pruneinf$, we get the language $\pref((\0\1)^*)$.
    The languages $L^{min}$ and $L^{max}$ are respectively $\0^*$ and $\pref((\0\1)^*)$, thus
    $\partial L = \0^* \cup \pref((\0\1)^*)$.
    We deduce that the boundary of $\Lambda_L$ is $\R_+\{(1,0), (\varphi,1)\}$, where $\varphi$ is the golden ratio.
    Indeed, we have $\0^n \R_+^2 \xrightarrow{n \to \infty} \R_+ (1,0)$ and
    $(\0\1)^n \R_+^2 \xrightarrow{n \to \infty} \R_+ (\varphi, 1)$.
    We obtain that the domain of state $0$ of the win-lose graph is the projective interval
    $\begin{pmatrix} 1 & \varphi \\ 0 & 1 \end{pmatrix} \R_+^2$.
    Thus, the invariant density at state $0$ is $f_0(x,y) = \frac{1}{x(\varphi x + y)}$.
\end{ex}

\section{Examples} \label{sec:exs}

In this section, we apply our algorithms to classical continued fraction algorithms
and some of their extensions.

\subsection{Example of dimension 1} \label{ss:ex:dim1}
The continued fraction algorithm
\[
    \begin{array}{rccl} F : &[0,1] &\to& [0,1] \\
                            & x & \mapsto & \left\{ \begin{array}{rcl} 2x &\text{ if }& x \in [0,1/2] \\
                                                                \frac{1-x}{x} &\text{ if }& x \in [1/2, 1]\end{array} \right.
    \end{array}
\]
has the same invariant density as the Euclid's algorithm.
Indeed, it can be described by the matrices graph with one vertex and edges labeled by matrices
$\begin{pmatrix} 1 & 0 \\ 1 & 2 \end{pmatrix}$ and $\begin{pmatrix}1 & 1 \\ 1 & 0\end{pmatrix}$.
More precisely, if $\tilde{F}$ is the algorithm of this matrices graph, we have $F \circ \varphi = \varphi \circ \tilde{F}$,
where $\varphi : \R_+^2 \to [0,1]$ is defined by $\varphi(x,y) = \frac{x}{x+y}$.
The algorithm of section~\ref{sec:algo:density} gives the domain $\begin{pmatrix} 2 & 1 \\ 1 & 1 \end{pmatrix} \R_+^2$
giving the density $\frac{1}{(2x + y)(x+y)}$ for the matrices graph.
The classical Euclid's algorithm is $E:[0,1] \to [0,1]$ such that
$E(x) = \{ \frac{1}{x} \}$ is the fractional part of $\frac{1}{x}$.
This Euclid's algorithm is described in the same way by the matrices graph with one vertex and edges labeled by matrices $\begin{pmatrix} 1 & 1 \\ n+1 & n \end{pmatrix}$, $n \in \N$, and we easily check that it has the same domain, 
thus it has the same invariant density.

\subsection{Example with a Cantor of singularities} \label{ex:cantor}

The win-lose graph of Figure~\ref{fig:cantor} defines a continued fraction algorithm
whose unique holomorphic extension of density at each state has a Cantor of singularities,
i.e., singularities form a compact set without isolated points.
See Figure~\ref{fig:domains:cantor} for an approximation of domains.
Indeed, we easily check that domain languages are of the form $\pref(A_i \{0,1\}^*)$ so
domains are the closure of a countable union of open intervals.
Thus, by Lemma~\ref{lem:C}, it suffices to prove that the boundary of each domain has no isolated point.
For each domain language $\D_i$,
we compute the language $\partial \D_i$ of Proposition~\ref{prop:boundary}.
The minimal automaton of $\partial \D_2$ is shown in Figure~\ref{fig:cantor:boundary}.
Then, if $\partial D_i = \Lambda_{\partial D_i}$ had an isolated point,
it would correspond to words that reach a terminal strongly connected component
of a pruned automaton recognizing $\partial \D_i$.
But we check that for every such a word of the form $w01^n$ (resp. $w10^n$),
we also have the word $w10^n$ (resp. $w01^n$) in $\partial \D_i$,
thus such a point is not isolated.

\begin{figure}
    \centering
    \includegraphics[width=.3\linewidth]{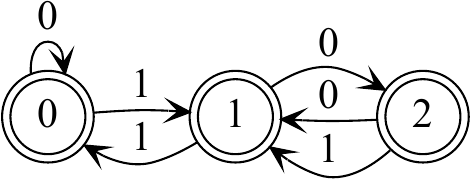}
    \caption{Example of win-lose graph where singularities of the holomorphic extension of the density is a Cantor set} \label{fig:cantor}
\end{figure}

\begin{figure}
    \centering
    \includegraphics[width=.9\linewidth]{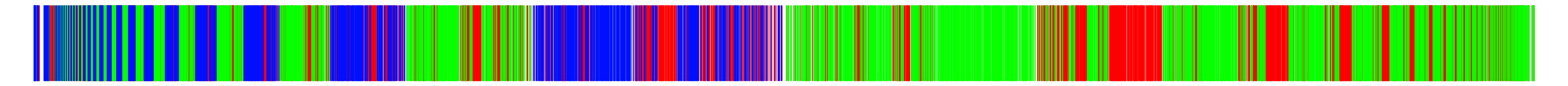}
    \caption{Approximation of domains for Example~\ref{ex:cantor}} \label{fig:domains:cantor}
\end{figure}

\begin{figure}
    \centering
    \includegraphics[width=.6\linewidth]{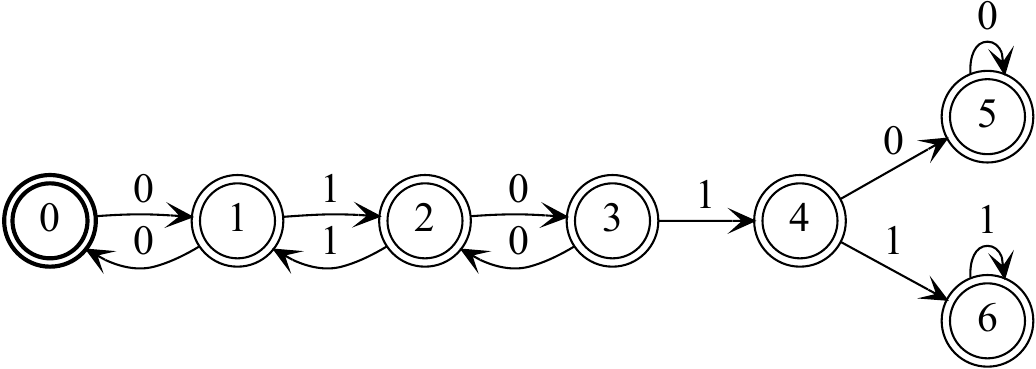}
    \caption{Minimal automaton of $\partial \D_2$ for Example~\ref{ex:cantor}} \label{fig:cantor:boundary}
\end{figure}

\subsection{Cassaigne} \label{ss:cassaigne}

The Cassaigne continued fraction algorithm is described by a matrices graph with a single state and with matrices
\[
    \left\{
    \left(\begin{array}{rrr}
    0 & 1 & 0 \\
    1 & 0 & 0 \\
    0 & 1 & 1
    \end{array}\right), \left(\begin{array}{rrr}
    1 & 1 & 0 \\
    0 & 0 & 1 \\
    0 & 1 & 0
    \end{array}\right)
    \right\}.
\]

We found in Subsection~\ref{ss:cp} that $\left(\begin{array}{rrr}
0 & 1 & 1 \\
1 & 1 & 1 \\
1 & 0 & 1
\end{array}\right)\R_+^3$ is a domain, thus an invariant density is
\[
    \frac{1}{{\left(x_{0} + x_{1} + x_{2}\right)} {\left(x_{0} + x_{1}\right)} {\left(x_{1} + x_{2}\right)}}.
\]

\begin{figure}
    \centering
    \includegraphics[width=.3\linewidth]{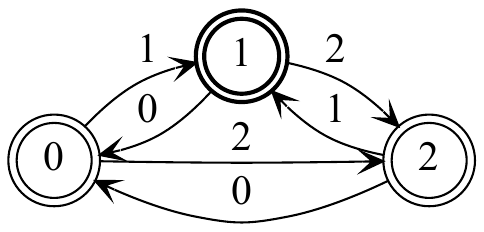}
    \caption{Win-lose graph for the Cassaigne algorithm} \label{fig:Cassaigne:win-lose}
\end{figure}

The Cassaigne continued fraction algorithm can be slowed down to the win-lose graph of Figure~\ref{fig:Cassaigne:win-lose}, with initial state $1$.
The invariant densities for this win-lose graph are
\begin{eqnarray*}
    f_0(x_0,x_1,x_2) = \frac{1}{{\left(x_{0} + x_{1} + x_{2}\right)} {\left(x_{0} + x_{1}\right)} {\left(x_{0} + x_{2}\right)}}, \\
    f_1(x_0,x_1,x_2) = \frac{1}{{\left(x_{0} + x_{1} + x_{2}\right)} {\left(x_{0} + x_{1}\right)} {\left(x_{1} + x_{2}\right)}}, \\
    f_2(x_0,x_1,x_2) = \frac{1}{{\left(x_{0} + x_{1} + x_{2}\right)} {\left(x_{0} + x_{2}\right)} {\left(x_{1} + x_{2}\right)}}.
\end{eqnarray*}

Thanks to Fougeron's criterion (see~\cite{Fougeron}), we can check that this algorithm is ergodic.
This algorithm is almost auto-dual: in restriction to the domain, the dual is the Cassaigne's algorithm, up to permutation.

\subsection{Brun} \label{ss:brun}

The Brun continued fraction algorithm subtracts the second greatest coordinate from the greatest one.
This is not directly described by a matrices graph, but we can convert it
to a matrices graph thanks to the algorithm described in Section~\ref{sec:general_to_matrices}.
For $d=3$, we obtain the matrices graph shown in Figure~\ref{fig:BrunMatrices}.

\begin{figure}
    \centering
    \includegraphics[width=.9\linewidth]{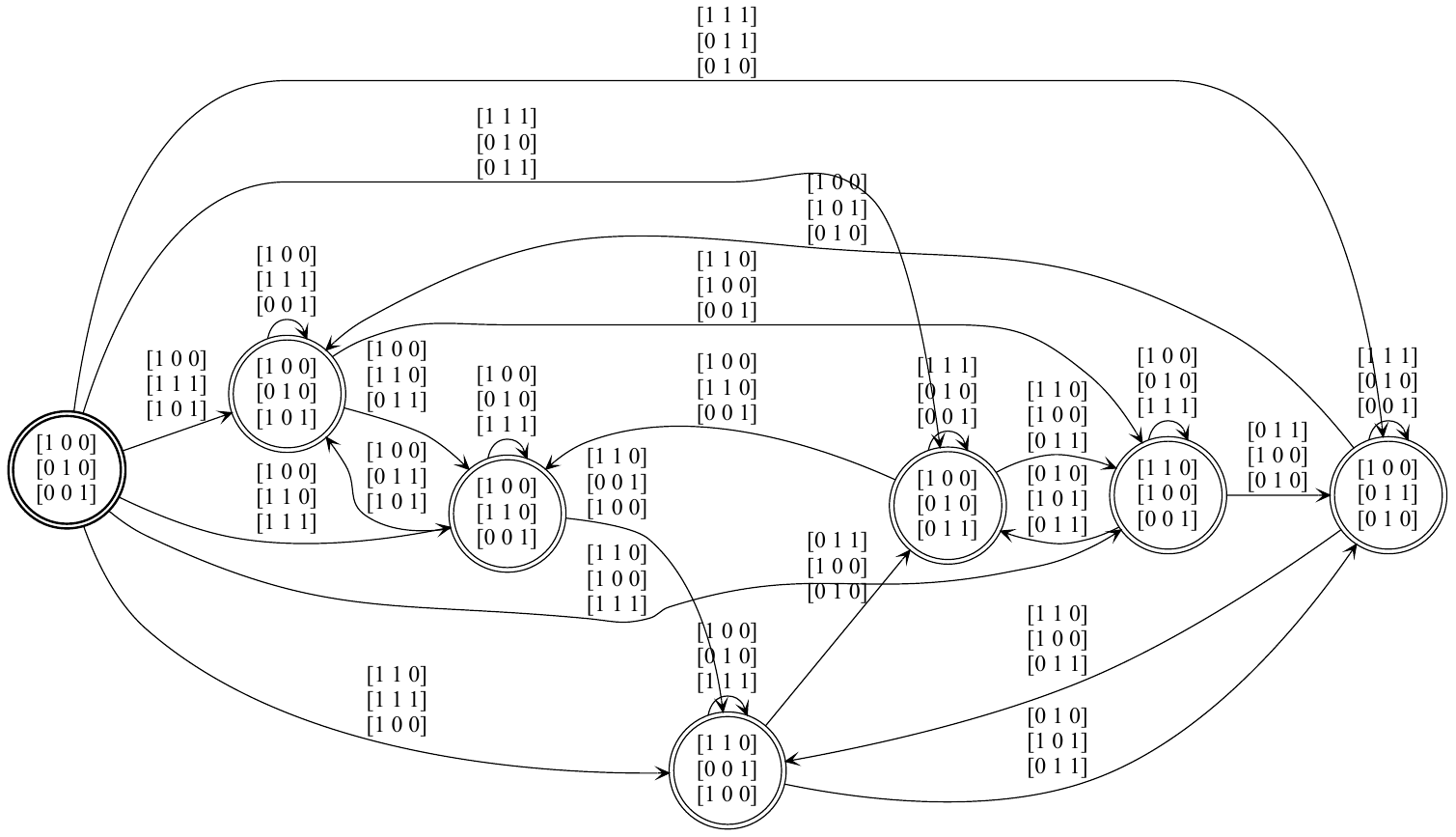}
    \caption{Matrices graph describing the Brun algorithm for $d=3$} \label{fig:BrunMatrices}
\end{figure}

Then, thanks to the algorithm of Subsection~\ref{ss:cp} we can compute the domains. For $d=3$, we get
\[
    \emptyset,
    \left(\begin{array}{rrr}
    1 & 1 & 2 \\
    1 & 1 & 1 \\
    0 & 1 & 1
    \end{array}\right) \R_+^3, \left(\begin{array}{rrr}
    1 & 1 & 2 \\
    1 & 1 & 1 \\
    0 & 1 & 1
    \end{array}\right) \R_+^3, \left(\begin{array}{rrr}
    0 & 1 & 1 \\
    1 & 1 & 2 \\
    1 & 1 & 1
    \end{array}\right) \R_+^3, \left(\begin{array}{rrr}
    1 & 1 & 2 \\
    1 & 1 & 1 \\
    0 & 1 & 1
    \end{array}\right) \R_+^3,
\]
\[
    \left(\begin{array}{rrr}
    0 & 1 & 1 \\
    1 & 1 & 2 \\
    1 & 1 & 1
    \end{array}\right) \R_+^3, \left(\begin{array}{rrr}
    1 & 1 & 2 \\
    0 & 1 & 1 \\
    1 & 1 & 1
    \end{array}\right) \R_+^3,
\]
thus, we get the densities for the matrices graph
\begin{eqnarray*}
    \frac{1}{{\left(2 \, x_{0} + x_{1} + x_{2}\right)} {\left(x_{0} + x_{1} + x_{2}\right)} {\left(x_{0} + x_{1}\right)}},
    \frac{1}{{\left(2 \, x_{0} + x_{1} + x_{2}\right)} {\left(x_{0} + x_{1} + x_{2}\right)} {\left(x_{0} + x_{1}\right)}},  \\
    \frac{1}{{\left(x_{0} + 2 \, x_{1} + x_{2}\right)} {\left(x_{0} + x_{1} + x_{2}\right)} {\left(x_{1} + x_{2}\right)}},
    \frac{1}{{\left(2 \, x_{0} + x_{1} + x_{2}\right)} {\left(x_{0} + x_{1} + x_{2}\right)} {\left(x_{0} + x_{1}\right)}}, \\
    \frac{1}{{\left(x_{0} + 2 \, x_{1} + x_{2}\right)} {\left(x_{0} + x_{1} + x_{2}\right)} {\left(x_{1} + x_{2}\right)}},
    \frac{1}{{\left(2 \, x_{0} + x_{1} + x_{2}\right)} {\left(x_{0} + x_{1} + x_{2}\right)} {\left(x_{0} + x_{2}\right)}}.
\end{eqnarray*}
We deduce the invariant density for the original algorithm by Lemma~\ref{lem:general:density}:
\[
    f(x) = \sum_{m \in \M,\ x \in m \R_+^d} f_m(m^{-1}x),
\]
where $\M \subset M_d(\R)$ is the set of states of the matrices graph.
For $d=3$, and for $x_0 < x_1 < x_2$, we get
\[
    f(x_0, x_1, x_2) = \frac{1}{(x_0 + x_2) x_1 x_2}.
\]

\begin{figure}
    \centering
    \includegraphics[width=.8\linewidth]{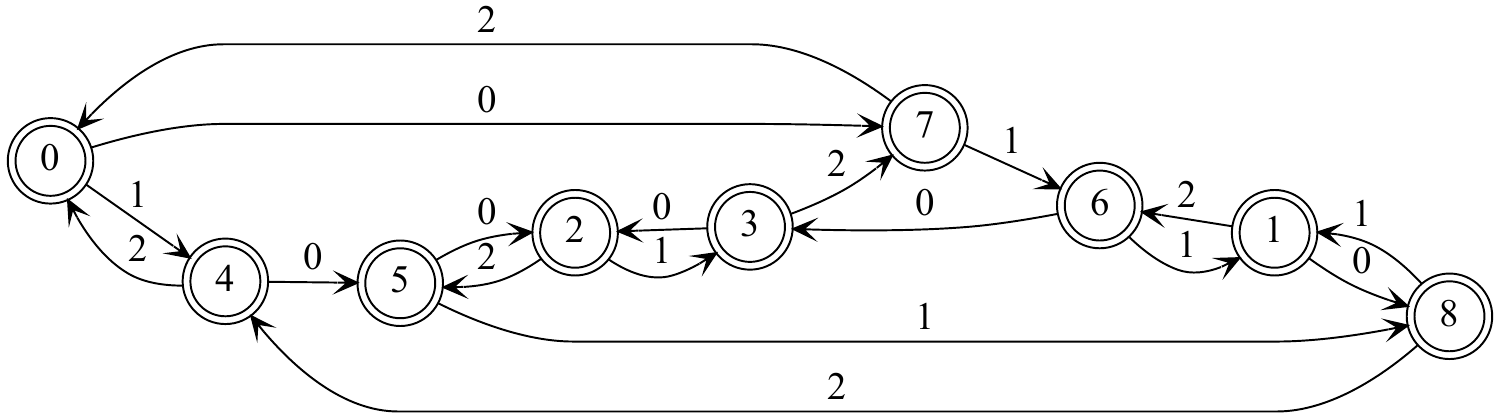}
    \caption{Win-lose graph of the Brun algorithm for $d=3$} \label{fig:Brun_winlose}
\end{figure}
\begin{figure}
    \centering
    \includegraphics[width=.5\linewidth]{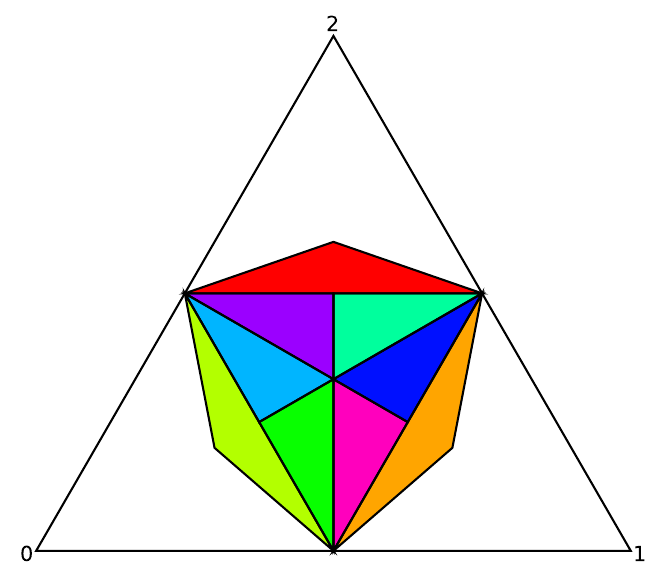}
    \caption{Domains of the win-lose graph of the Brun algorithm for $d=3$} \label{fig:Brun_domains}
\end{figure}

The matrices graph can be decomposed into a win-lose graph.
Figure~\ref{fig:Brun_winlose} shows the strongly connected component of this graph for $d=3$.
Thanks to Fougeron's criterion, we can check that this algorithm is ergodic for every $d$ (see~\cite{Fougeron}).

\subsection{Poincare}

The Poincaré algorithm subtracts the second greatest coordinate from the greatest, the third greatest from the second, etc...
For $d=3$, it is defined by the matrices graph with one state and with matrices
\[
    \left(\begin{array}{rrr}
    1 & 0 & 0 \\
    1 & 1 & 0 \\
    1 & 1 & 1
    \end{array}\right), \left(\begin{array}{rrr}
    1 & 1 & 1 \\
    0 & 1 & 1 \\
    0 & 0 & 1
    \end{array}\right), \left(\begin{array}{rrr}
    1 & 0 & 0 \\
    1 & 1 & 1 \\
    1 & 0 & 1
    \end{array}\right),
\]
\[
    \left(\begin{array}{rrr}
    1 & 1 & 0 \\
    0 & 1 & 0 \\
    1 & 1 & 1
    \end{array}\right), \left(\begin{array}{rrr}
    1 & 0 & 1 \\
    1 & 1 & 1 \\
    0 & 0 & 1
    \end{array}\right), \left(\begin{array}{rrr}
    1 & 1 & 1 \\
    0 & 1 & 0 \\
    0 & 1 & 1
    \end{array}\right)
\]

Since the set of matrices is stable by transposition,
the algorithm is auto-dual,
and the full positive cone $\R_+^d$ is a domain,
thus $\frac{1}{x_0 ... x_{d-1}}$ is an invariant density.

For $d=3$, the algorithm can be described by the win-lose graph in Figure~\ref{fig:Poincare_winlose},
and we can check that it is not ergodic and not convergent (see~\cite{Nogueira}).
For $d=4$, it can be described by a win-lose graph with $20$ states, but it doesn't satisfy the Fougeron's criterion,
and it is an open question to determine whether it is ergodic.

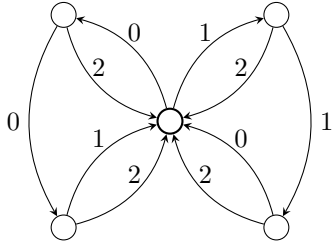
\begin{figure}
    \centering
    \begin{tikzpicture}[node distance=2cm,-stealth]
        \node[state, thick] (q0) {};
        \node[state, above right of=q0] (q1) {};
        \node[state, below right of=q0] (q2) {};
        \node[state, above left of=q0] (q3) {};
        \node[state, below left of=q0] (q4) {};
        
        \draw (q0) edge[bend left] node[above]{$1$} (q1)
              (q0) edge[bend right] node[above]{$0$} (q3)
              (q1) edge[bend left] node[above]{$2$} (q0)
              (q1) edge[bend left] node[right]{$1$} (q2)
              (q2) edge[bend left] node[above]{$2$} (q0)
              (q2) edge[bend right] node[above]{$0$} (q0)
              (q3) edge[bend right] node[above]{$2$} (q0)
              (q3) edge[bend right] node[left]{$0$} (q4)
              (q4) edge[bend left] node[above]{$1$} (q0)
              (q4) edge[bend right] node[above]{$2$} (q0);
    \end{tikzpicture}
    \caption{Poincaré algorithm as a win-lose graph for $d=3$} \label{fig:Poincare_winlose}
\end{figure}

\subsection{Reverse}

The reverse algorithm is defined as the Arnoux-Rauzy's one if one coordinate is greater than the sum of the others,
and it sends the remaining center in the entire positive cone.
It is given by the matrices graph with one state and matrices
\[
    \left(\begin{array}{rrr}
    1 & 0 & 0 \\
    0 & 1 & 0 \\
    1 & 1 & 1
    \end{array}\right), \left(\begin{array}{rrr}
    1 & 0 & 0 \\
    1 & 1 & 1 \\
    0 & 0 & 1
    \end{array}\right), \left(\begin{array}{rrr}
    1 & 1 & 1 \\
    0 & 1 & 0 \\
    0 & 0 & 1
    \end{array}\right),
    \left(\begin{array}{rrr}
    0 & 1 & 1 \\
    1 & 0 & 1 \\
    1 & 1 & 0
    \end{array}\right).
\]

By the algorithm of Subsection~\ref{ss:cp} we find the domain $\left(\begin{array}{rrr}
0 & 1 & 1 \\
1 & 0 & 1 \\
1 & 1 & 0
\end{array}\right) \R_+^3$, thus an invariant density is
\[
    f(x_0,x_1,x_2) = \frac{2}{{\left(x_{0} + x_{1}\right)} {\left(x_{0} + x_{2}\right)} {\left(x_{1} + x_{2}\right)}}.
\]

This algorithm is almost auto-dual: in restriction to the domain, the dual is the reverse algorithm.
It cannot be decomposed as a win-lose graph due to one of the matrices having a determinant of $2$.

\subsection{Fully subtractive}

The fully subtractive algorithm subtracts the smallest coordinate from every other one.
It is described by the win-lose graph with one state and $d$ letters.

The unique domain is the Rauzy gasket for $d=3$ and a generalization of it if $d \geq 4$.
It has zero-Lebesgue measure (see~\cite{AHS} for more details for $d=3$), thus we cannot find
an invariant density by this method.
For $d \geq 3$, this algorithm is neither ergodic nor convergent.
The dual of this algorithm is the Arnoux-Rauzy's one, which subtracts from the greatest coordinate the sum of the others.

\subsection{Jacobi-Perron}

The Jacobi-Perron continued fraction algorithm subtracts, as many times as possible, the first coordinate from the other ones and then puts this first coordinate in the last position.
For example, for $d=3$, the algorithm is $(x,y,z) \mapsto (y - \floor{\frac{y}{x}} x,\ z - \floor{\frac{z}{x}} x,\ x)$, where $\floor{\cdot}$ denotes the floor function.

\begin{figure}
    \centering
    \includegraphics[width=.25\linewidth]{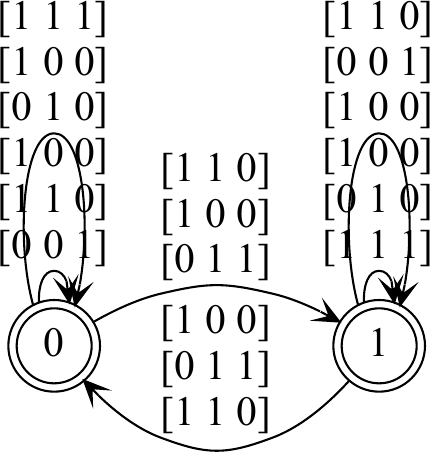}
    \caption{A matrices graph for the Jacobi-Perron algorithm for $d=3$} \label{fig:JP_matrices}
\end{figure}

\begin{figure}
    \centering
    \includegraphics[width=.75\linewidth]{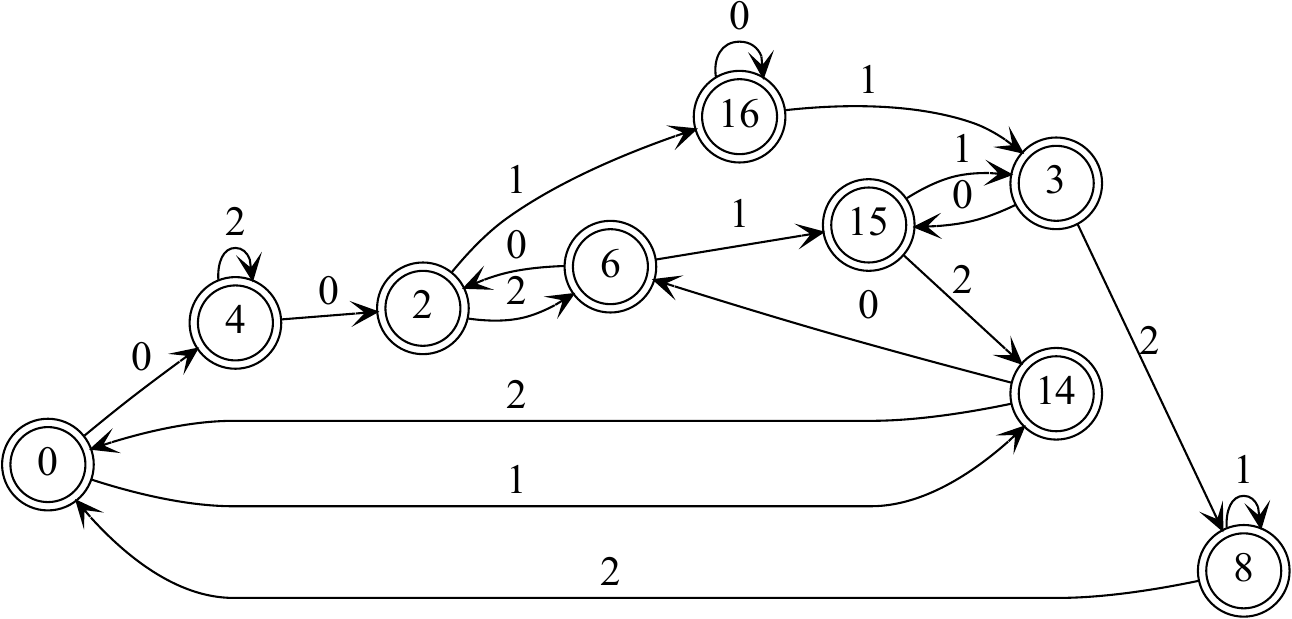}
    \caption{Win-lose graph for the Jacobi-Perron algorithm for $d=3$} \label{fig:JP_winlose}
\end{figure}

This algorithm can be slowed down to be described by a matrices graph.
For $d=3$, we obtain a matrices graph with $4$ states, and its main strongly connected component is depicted in Figure~\ref{fig:JP_matrices}.
Moreover, we can decompose this strongly connected component into the win-lose graph
shown in Figure~\ref{fig:JP_winlose}.
Thanks to Fougeron's criterion, we can check that this algorithm is ergodic for $d = 3$.

The invariant density for this algorithm is unknown.
The domain is fractal, and it is unclear whether it has zero Lebesgue measure.
See Figure~\ref{fig:JP_domain} for an approximation of the domains of the win-lose graph.

\begin{figure}
    \centering
    \includegraphics[width=.5\linewidth]{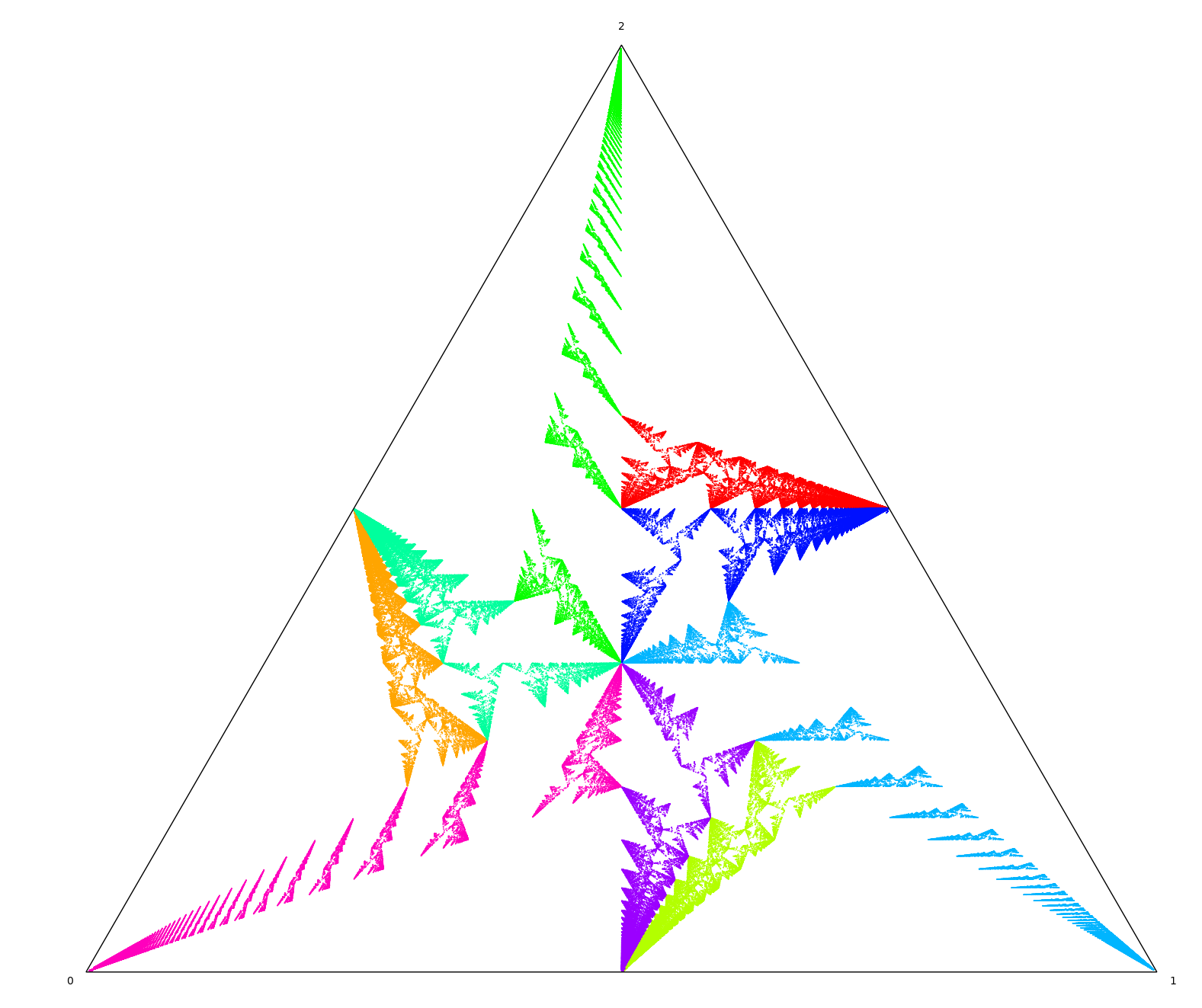}
    \caption{Approximation of domains of the win-lose graph for the Jacobi-Perron algorithm for $d=3$} \label{fig:JP_domain}
\end{figure}

\subsection{Symmetric Jacobi-Perron}

The Symmetric Jacobi-Perron continued fraction algorithm subtracts as many times as possible
the smallest coordinate from the other ones.

\begin{figure}
    \centering
    \includegraphics[width=.25\linewidth]{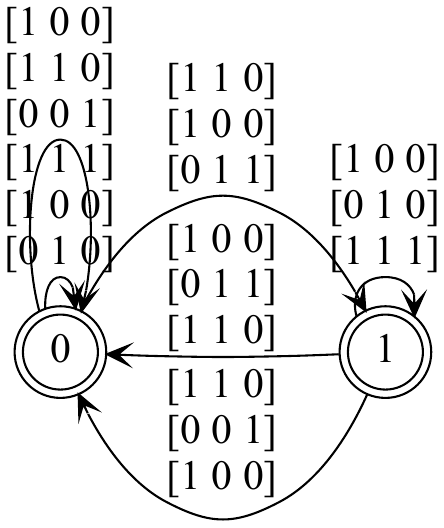}
    \caption{A matrices graph for the Symmetric Jacobi-Perron algorithm for $d=3$} \label{fig:SJP_matrices}
\end{figure}

For $d=3$, we can describe a slowed-down version of this algorithm by a matrices graph with $4$ states, and its main strongly connected component is depicted in Figure~\ref{fig:SJP_matrices}.
Moreover, we can decompose this strongly connected component into the win-lose graph
of Figure~\ref{fig:SJP_winlose}.
Thanks to Fougeron's criterion, we can verify that this algorithm is ergodic for $d = 3$.

\begin{figure}
    \centering
    \includegraphics[width=\linewidth]{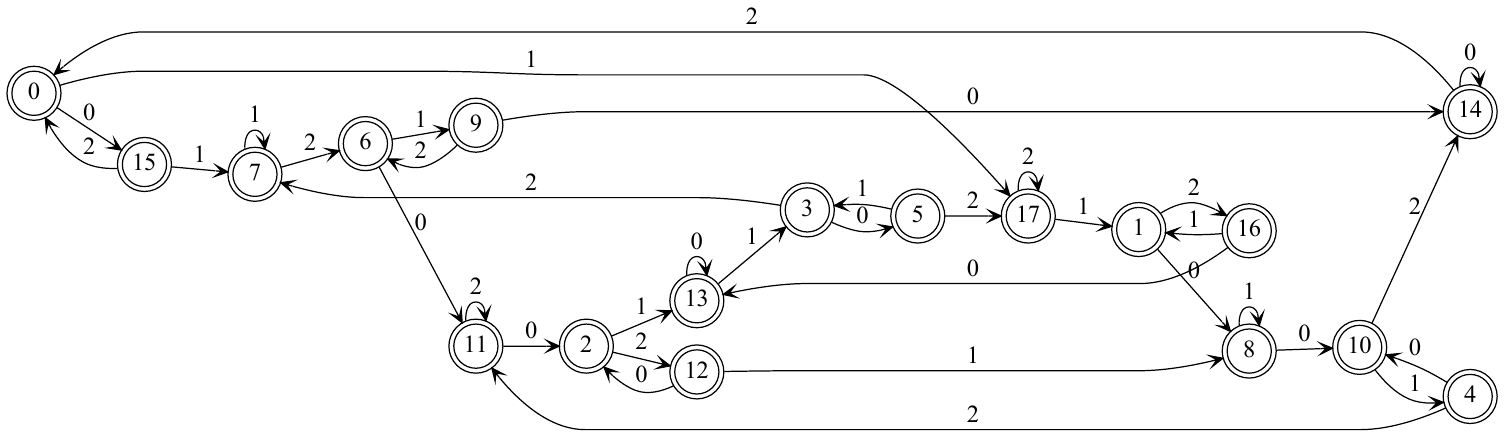}
    \caption{Win-lose graph for the Symmetric Jacobi-Perron algorithm for $d=3$} \label{fig:SJP_winlose}
\end{figure}

The invariant density for this algorithm is unknown.
The domain is fractal, and we don't know whether it has zero Lebesgue measure.
See Figure~\ref{fig:SJP_domain} for an approximation of the domains of the win-lose graph.

\begin{figure}
    \centering
    \includegraphics[width=.5\linewidth]{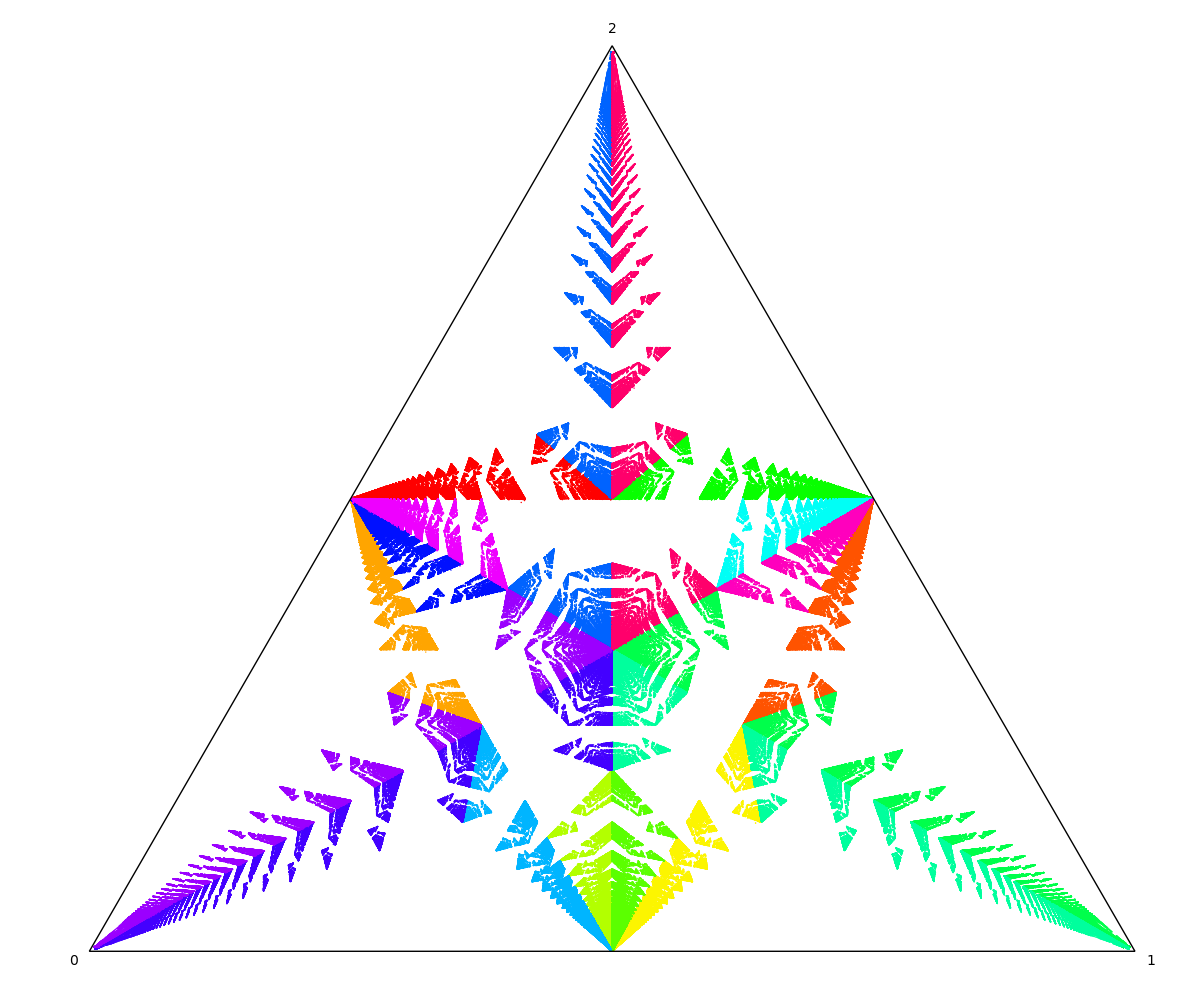}
    \caption{Approximation of domains of the win-lose graph for the Symmetric Jacobi-Perron algorithm for $d=3$} \label{fig:SJP_domain}
\end{figure}

\subsection{Arnoux-Rauzy-Poincaré}

The Arnoux-Rauzy-Poincaré continued fraction algorithm is a combination of Arnoux-Rauzy and Poincaré's one.
We apply the Arnoux-Rauzy algorithm if possible (i.e. we subtract from the greatest coordinate the sum of the others),
otherwise we apply the Poincaré's one (i.e. we subtract from the second greatest coordinate the smallest, and from the greatest the second greatest).

\begin{figure}
    \centering
    \includegraphics[width=.8\linewidth]{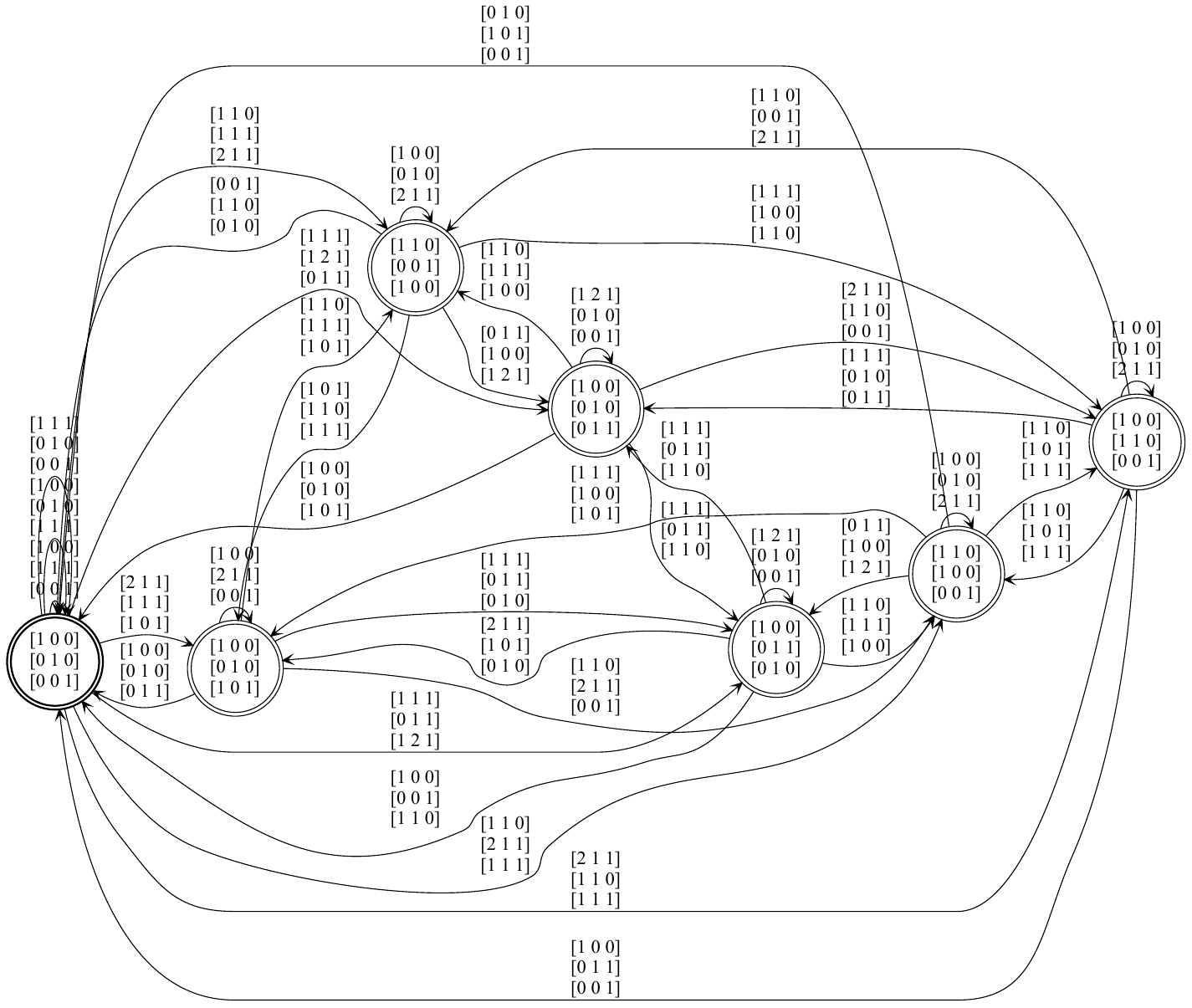}
    \caption{Matrices graph for the Arnoux-Rauzy-Poincaré algorithm} \label{fig:ARP_matrices}
\end{figure}

The Arnoux-Rauzy-Poincaré algorithm can be represented by a matrices graph,
thanks to the algorithm of Section~\ref{sec:general_to_matrices}.
We get the graph of Figure~\ref{fig:ARP_matrices}.

The invariant density is unknown.
The domains are fractal, and we don't know whether they have non-zero Lebesgue measure.

\subsection{Two letters win-lose graph with non-rational density}

The win-lose graph
\begin{center}
    \begin{tikzpicture}[node distance=2cm,->,-stealth]
        \node[state] (q0) {$0$};
        \node[state, above right of=q0] (q3) {$3$};
        \node[state, below right of=q3] (q1) {$1$};
        \node[state, right of=q1] (q2) {$2$};
        
        \draw (q0) edge[loop left] node[left]{$0$} (q0)
              (q0) edge node[above left]{$1$} (q3)
              (q1) edge node[below]{$0$} (q0)
              (q1) edge[bend left] node[above]{$1$} (q2)
              (q2) edge[loop right] node[right]{$1$} (q2)
              (q2) edge node[below]{$0$} (q1)
              (q3) edge[loop below] node[below]{$1$} (q3)
              (q3) edge node[above right]{$0$} (q1);
    \end{tikzpicture}
\end{center}
have domain languages $\D_0 = \0^2 \{\0,\1\}^*$, $\D_1 = \0\1 \{\0,\1\}^*$, $\D_2 = \1^+\0\1 \{\0,\1\}^*$
and $\D_3 = \1^+ \0^2 \{\0,\1\}^*$.
Thus, domains are $D_0 = \begin{pmatrix} 1 & 2 \\ 0 & 1 \end{pmatrix} \R_+^2$, $D_1 = \begin{pmatrix} 2 & 1 \\ 1 & 1 \end{pmatrix} \R_+^2$,
$D_2 = \bigcup_{n \geq 1} \begin{pmatrix} 2 & 1 \\ 2n+1 & n+1 \end{pmatrix} \R_+^2$ and
$D_3 = \bigcup_{n \geq 1} \begin{pmatrix} 1 & 2 \\ n & 2n+1 \end{pmatrix} \R_+^2$.
Furthermore, all these interval are pairwise Lebesgue-disjoint.
We get densities 
\begin{eqnarray*}
    f_0(x,y) &=& \frac{1}{x(2x+y)}, \quad
    f_1(x,y) = \frac{1}{(2x+y)(x+y)}, \\
    f_2(x,y) &=& \sum_{n \geq 1} \frac{1}{(2x+(2n+1)y)(x+(n+1)y)}, \\
    f_3(x,y) &=& \sum_{n \geq 1} \frac{1}{(x+ny)(2x+(2n+1)y)}.
\end{eqnarray*}
By Proposition~\ref{prop:non-rat}, $f_2$ and $f_3$ are not rational fractions.

The Figure~\ref{fig:cantor} is another example of win-lose graph with non-rational densities that cannot be made as explicit as here since they have Cantor of singularities.

\subsection{Other examples} \label{ss:sage}

More examples can be found here: \url{http://www.i2m.univ-amu.fr/perso/paul.mercat/ComputeInvariantDensities.pdf}.

Additionally, many other examples can be easily tested since the algorithms described in this article
are implemented in a package for the Sage math software (see \url{https://www.sagemath.org/}).
The package is freely available here: \url{https://gitlab.com/mercatp/badic}
and can be installed with the following command:
\begin{lstlisting}[language=bash]
  $ sage -pip install badic
\end{lstlisting}


\section{Construction of extensions from sets of quadratic numbers} \label{sec:quadratic}

In this section, we present an algorithm that takes a finite set of quadratics numbers as input
and outputs a win-lose graph on two letters with an invariant density
where the quadratic numbers appear.
We denote
$\0 = \begin{pmatrix} 1 & 1 \\ 0 & 1 \end{pmatrix}$
and $\1 = \begin{pmatrix} 1 & 0 \\ 1 & 1 \end{pmatrix}$ to lighten the notations.

The algorithm is as follows:
\begin{itemize}
    \item If the number of elements in the set is odd, add or remove the number $0$ to the set.
    Change signs to ensure non negative numbers.
    
    \item Compute the continued fraction expansion of $(x,1)$ for each quadratic number $x$,
    for the fully subtractive algorithm on two letters.
    Each expansion is of the form $uv^\omega$, where $u$ and $v$
    are two finite words over the alphabet $\{\0,\1\}$.
    
    \item For each quadratic number $x$ with expansion $uv^\omega$, compute
    the rational language $L_x$ of finite words less than $uv^\omega$ in lexicographical order.
    A deterministic automaton recognizing this language is easily computed by considering
    the minimal automaton recognizing the language $uv^*$, and then adding adding a new state
    $s$ with edges $s \xrightarrow{\0} s$ and $s \xrightarrow{\1} s$, and adding edges
    $t \xrightarrow{\0} s$ for each state $t$ that have no outgoing edge labeled by $\0$.
    \item If $x_0 < x_1 < ... < x_{2n+1}$ are the ordered quadratic numbers, compute the language
    \[
        L = \bigcup_{i \in \{0,...,n\}} L_{x_{2i}} \cap L_{x_{2i+1}}^c.
    \]
    
    \item Take the mirror $L^{mirror}$ (i.e. the language of words of $L$ in reverse order).
    
    \item Compute a deterministic automaton recognizing this mirror $L^{mirror}$.
    This deterministic automaton gives a win-lose graph on two letters.
\end{itemize}

\begin{ex} \label{ex:sqrt2}
    Consider the set $\{ 0, \sqrt{2} \}$.
    The expansion of $(0,1)$ is $\1^\omega$, and the expansion of $(\sqrt{2},1)$ is $(\0\1\1\0)^\omega$.
    Then, the rational language $L_0$ is $\{\0,\1\}^*$, and the rational language
    $L_{\sqrt{2}}$ is recognized by the automaton shown in Figure~\ref{fig:automaton_sqrt2}.
    The mirror of the language $L = L_0 \cap L_{\sqrt{2}}^c$ is recognized by the deterministic
    automaton depicted in Figure~\ref{fig:automaton_0sqrt2}.
    It is a win-lose graph whose domains are projective intervals between points
    \[
        \{(0:1),(\sqrt{2}-1:1),(\sqrt{2}-1:2-\sqrt{2}),(\sqrt{2}:1),(\sqrt{2}-1:3-2\sqrt{2}),(1:0)\}.
    \]
    Then, we can easily deduce the invariant density, where $\sqrt{2}$ appears.
\end{ex}

\begin{rem}
    More examples can be found here:
    \url{http://www.i2m.univ-amu.fr/perso/paul.mercat/ComputeInvariantDensities.html}.
    And any example can be easily computed since this algorithm is implemented in the Sage mathematical software using the badic package,
    see Subsection~\ref{ss:sage} for more details.
\end{rem}

\begin{figure}
    \centering
    \includegraphics[width=.4\linewidth]{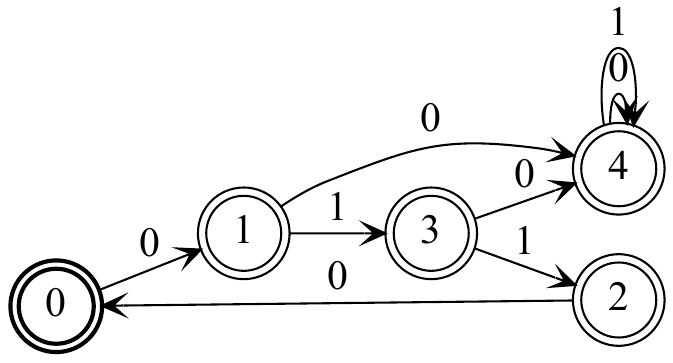}
    \caption{Automaton recognizing $L_{\sqrt{2}}$} \label{fig:automaton_sqrt2}
\end{figure}

\begin{figure}
    \centering
    \includegraphics[width=.5\linewidth]{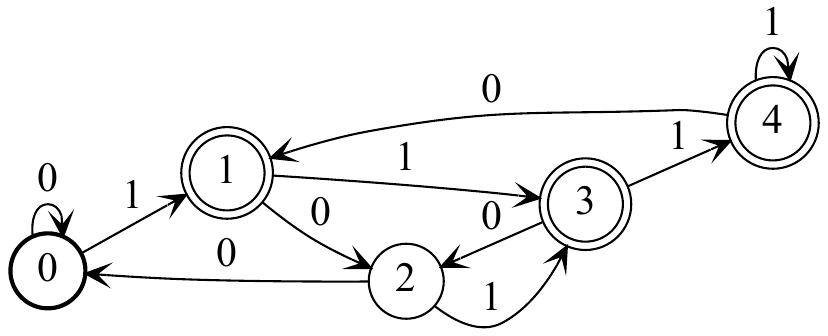}
    \caption{Automaton recognizing the language $L^{mirror}$ of Example~\ref{ex:sqrt2}} \label{fig:automaton_0sqrt2}
\end{figure}

\begin{prop}
    The algorithm above gives a win-lose graph whose domains are finite unions of projective intervals,
    and $(x,1)$ is in the boundary for every quadratic number $x$ in the input.
\end{prop}

This proposition follows from the following results.

\begin{lemme}
    $\Lambda_L$ is the union of intervals between $(x_{2i}:1)$ and $(x_{2i+1}:1)$.
\end{lemme}

\begin{proof}
    By construction, $\Lambda_{L_x}$ is the projective interval between $(x:1)$ and $(1:0)$.
    Thus the limit set of $L_{x_{2i}} \cap L_{x_{2i+1}}^c$ is
    the projective interval between $(x_{2i}:1)$ and $(x_{2i+1}:1)$.
\end{proof}

\begin{prop} \label{prop:wl:mirror}
    Let $L$ be a regular language. 
    Then, a pruned automaton recognizing
    $L^{mirror}$ 
    is a win-lose graph,
    satisfying $\bigcup_{i \in F} \D_i = \pref(L)$,
    where $F$ is its set of final states.
\end{prop}

\begin{proof}
    By definition, a word of language $\D_i$ is a label of path toward state $i$ in this win-lose graph.
    Thus, $\bigcup_{i \in F} \D_i$ is the set of labels of paths toward states in $F$.
    Words of $L$ are paths from the initial state, toward states in $F$.
    Since the automaton is pruned, prefixes of words of $L$ are exactly labels of paths toward states in $F$.
\end{proof}

\begin{rem}
    Proposition~\ref{prop:wl:mirror} can be used to construct many examples of interesting win-lose graphs.
    For example consider the automaton
    \vspace{-.2cm}
    \begin{center}
        \begin{tikzpicture}[node distance=2cm,-stealth]
            \node[state, thick, double] (q0) {};
            \node[state, double, right of=q0] (q1) {};
            \node[state, double, right of=q1] (q2) {};
            \node[state, double, right of=q2] (q3) {};
            
            \draw (q0) edge[loop above] node[above]{$\0$} (q0)
                  (q0) edge node[above]{$\1$} (q1)
                  (q1) edge[loop above, -stealth] node[above]{$\0$} (q1)
                  (q1) edge node[above]{$\1$} (q2)
                  (q2) edge node[above]{$\1$} (q3)
                  (q3) edge[loop right, out=40, in=-40, looseness=18,-stealth] node[right]{$\0$} (q3)
                  (q3) edge[loop right, -stealth] node[right]{$\1$} (q3);
        \end{tikzpicture}
    \end{center}
    \vspace{-.7cm}
    Its language $L$ is stable by prefixes.
    By construction, its limit set $\Lambda_L$ has a boundary with infinitely many accumulation points.
    The mirror of $L$ gives the win-lose graph of Figure~\ref{fig:wl:ac},
    and we have $D_0 \cup D_1 \cup D_2 \cup D_3 = \Lambda_L$ since $\D_0 \cup \D_1 \cup \D_2 \cup \D_3 = L$.
\end{rem}

\begin{figure}
    \centering
    \includegraphics[width=.8\linewidth]{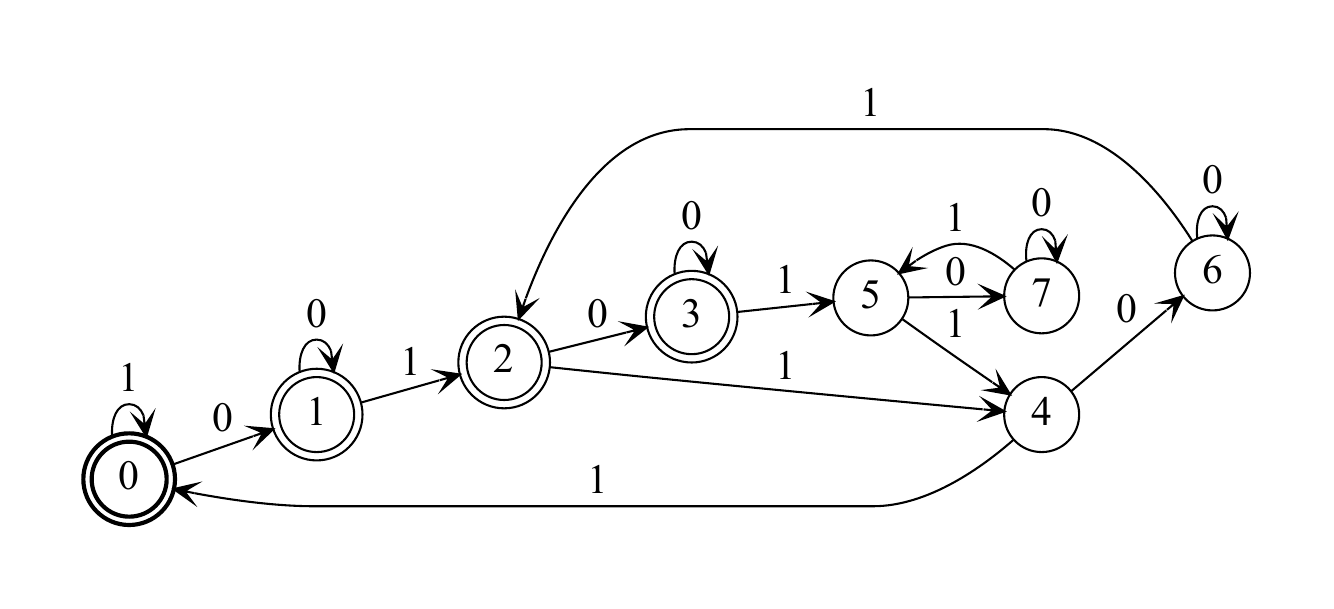}
    \caption{Win-lose graph with domains whose boundaries have infinitely many accumulation points} \label{fig:wl:ac}
\end{figure}

\begin{figure}
    \centering
    \includegraphics[width=.32\linewidth]{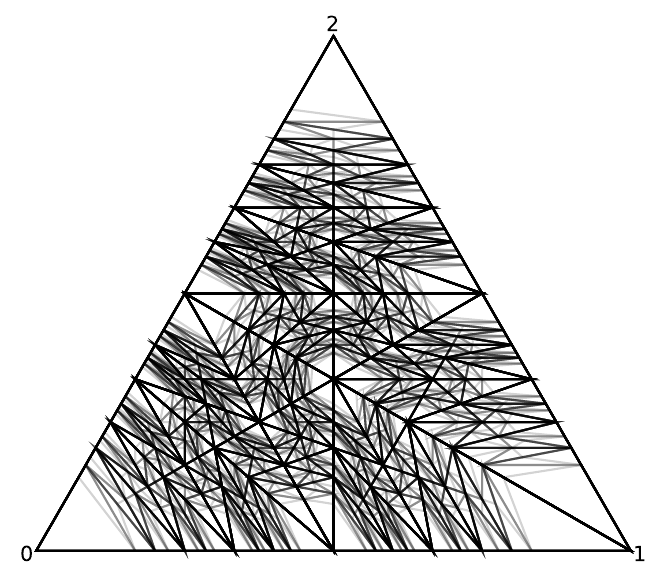}
    \includegraphics[width=.32\linewidth]{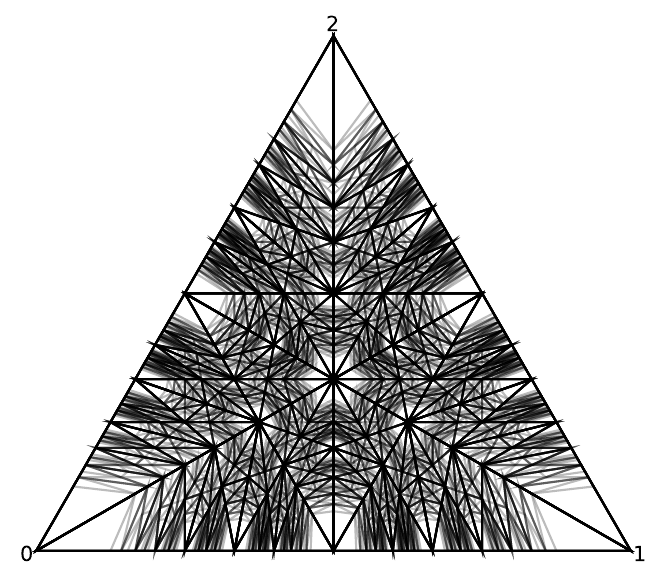}
    \includegraphics[width=.32\linewidth]{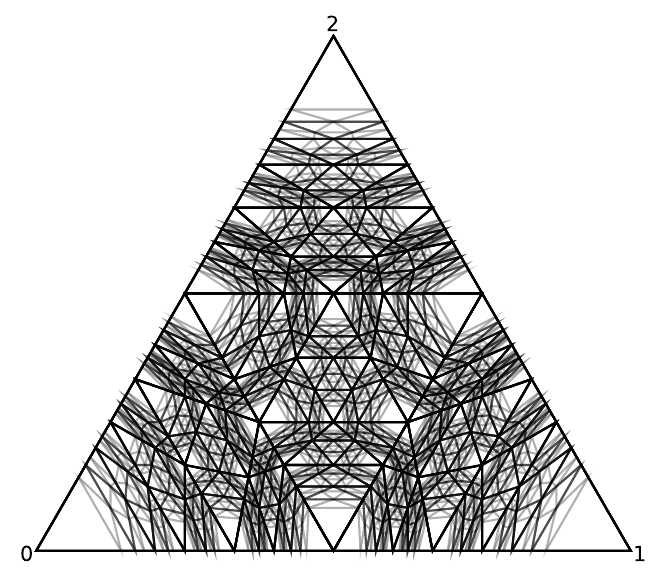}
    \includegraphics[width=.32\linewidth]{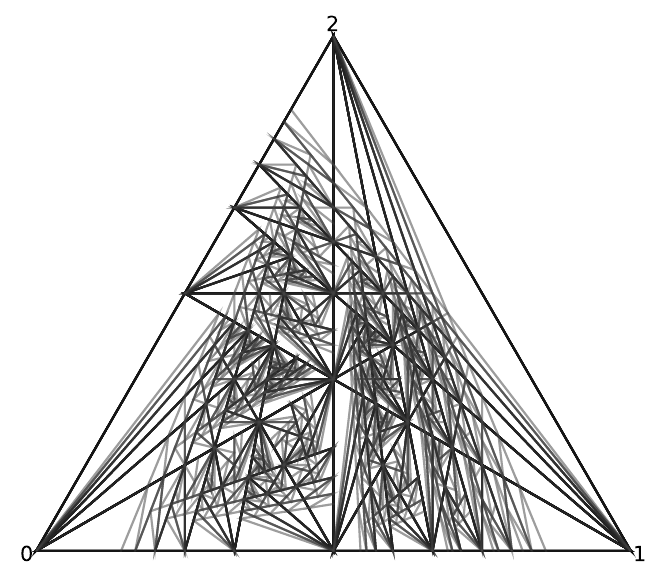}
    \includegraphics[width=.32\linewidth]{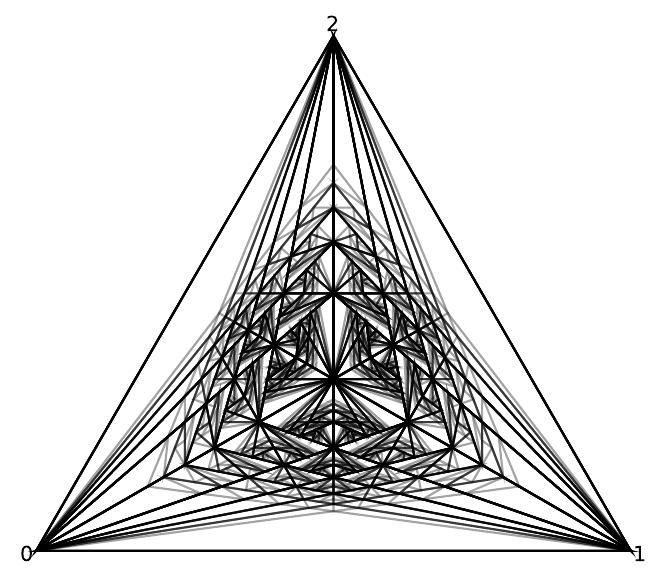}
    \includegraphics[width=.32\linewidth]{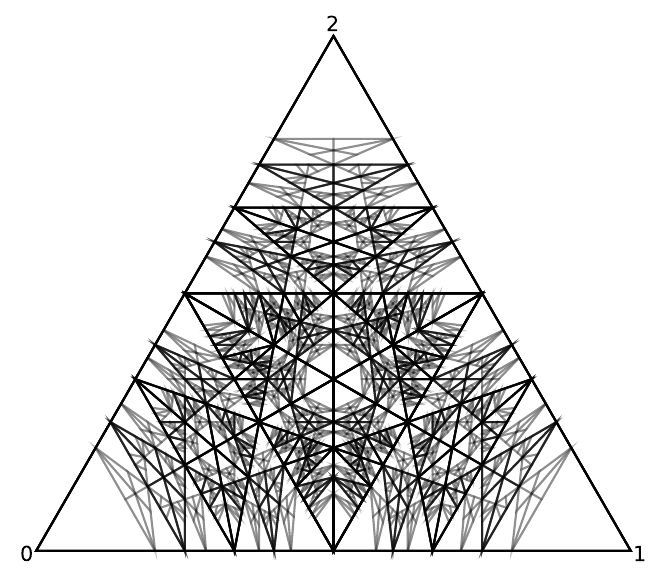}
    \includegraphics[width=.32\linewidth]{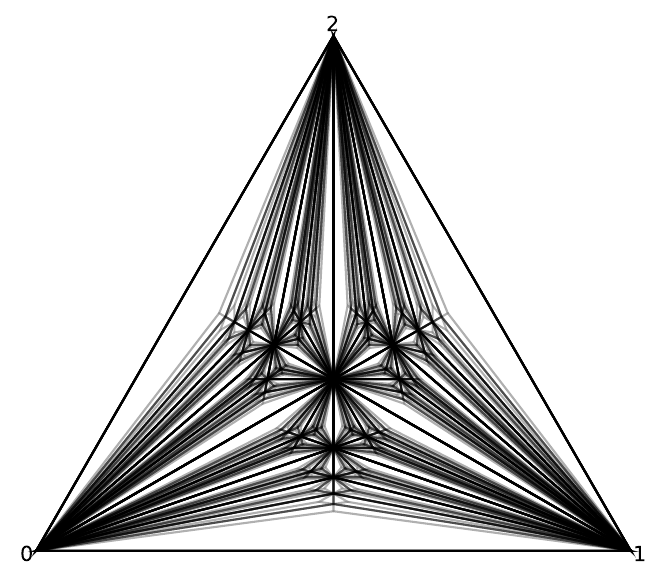}
    \includegraphics[width=.32\linewidth]{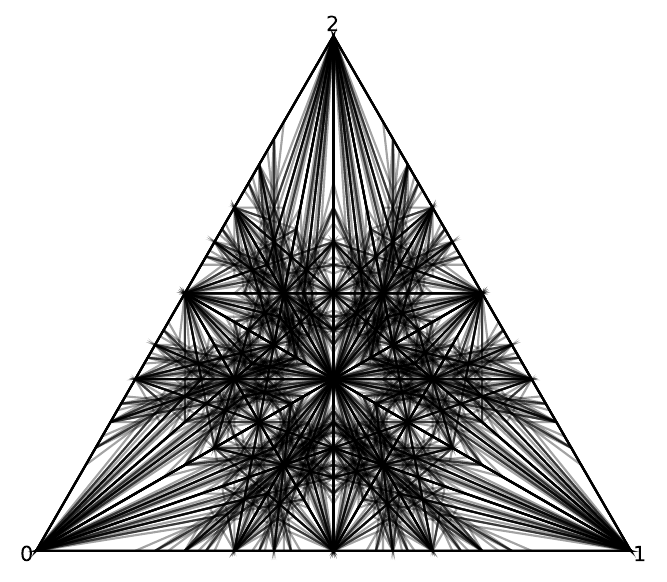}
    \caption{Some steps of Cassaigne, Brun, reverse, Jacobi-Perron, symmetric Jacobi-Perron, Arnoux-Rauzy-Poincaré, fully subtractive, and Poincaré's algorithms}
\end{figure}

\section{Acknowledgments}

I thank Charles Fougeron and Vincent Delecroix for interesting discussions. Without them, this article wouldn't exist.
I also thank Pierre Arnoux for interesting discussions and for the example of Subsection~\ref{ss:ex:dim1}.

\end{document}